\documentclass{article}
\usepackage[centertags]{amsmath}
\usepackage{amsfonts}
\usepackage{amssymb}
\usepackage{amsthm}
\usepackage{booktabs}
\usepackage{mathtools} 
\usepackage{graphicx}
\usepackage{xcolor}

\theoremstyle{plain} 
 \theoremstyle{definition}
\theoremstyle{definition}

\theoremstyle{plain}
\newtheorem{thm}{Theorem}

\theoremstyle{plain}

\theoremstyle{plain}
\newtheorem{prop}{Proposition}

\theoremstyle{plain}

\theoremstyle{plain}

\theoremstyle{definition}
\newtheorem{exmp}{Example}

\theoremstyle{remark}
\newtheorem{rem}{Remark}

\newcommand{\R}{\mathbb{R}}

\newcommand{\C}{\mathbb{C}}

\newcommand{\PDO}{\Psi\textrm{DO}}
\newcommand{\Ord}{\mathcal{O}}

\newcommand{\del}{\partial}

\newcommand{\ii}{i}

\newcommand{\Arg}{\operatorname{Arg}}
\newcommand{\erfc}{\operatorname{erfc}}

\newcommand{\sgn}{\operatorname{sgn}}
\newcommand{\re}{\operatorname{Re}}

\newcommand{\relErr}{\varepsilon_\textit{rel}}
\newcommand{\res}{r}
\newcommand{\ub}{\varrho}

\newcommand{\wt}{\widetilde}

\newcommand{\Tinf}{T_{f}}

\newcommand{\Ntau}{N_{\tau}}
\newcommand{\Nspc}{N_{x}}

\newcommand{\bx}{\boldsymbol x}

\newcommand{\bu}{\boldsymbol u}
\newcommand{\bp}{\boldsymbol p}
\newcommand{\bxi}{\boldsymbol \xi}
\newcommand{\bk}{\boldsymbol \kappa}
\newcommand{\bD}{\boldsymbol D}
\newcommand{\bE}{\boldsymbol E}
\newcommand{\bH}{\boldsymbol H}
\newcommand{\bB}{\boldsymbol B}

\newcommand{\bEh}{\widehat\bE}

\newcommand{\Dom}{D}
\newcommand{\Bdy}{\del\Dom}
\newcommand{\vInc}{v^{i}}
\newcommand{\vSca}{v^{s}}
\newcommand{\vTot}{v}
\newcommand{\refco}{m}

\newcommand{\A}{\mathcal A}
\newcommand{\B}{\mathcal B}
\newcommand{\CC}{\mathcal C}
\newcommand{\I}{\mathcal I}
\newcommand{\Id}{I}

\newcommand{\srHel}{\left[ \Id + \Delta / \kappa^2 \right]^{1/2}}
\newcommand{\isrHel}{\left[ \Id + \Delta / \kappa^2 \right]^{-1/2}}

\newcommand{\iHel}{\left[ \Id + \Delta / \kappa^2 \right]^{-1}}
\newcommand{\Helmu}{\left[ \refco(\bx) + \frac{\Delta}{\kappa^2} \right]}
\newcommand{\isrHelmu}{\left[ \refco(\bx) + \frac{\Delta}{\kappa^2} \right]^{-1/2}}
\newcommand{\unitInt}{[-1,1]}

\newcommand{\intinf}{\int_{-\infty}^{\infty}}
\newcommand{\dblintinf}{\int_{-\infty}^{\infty}\int_{-\infty}^{\infty}}

\newcommand{\intzeroinf}{\int_0^{\infty}}
\newcommand{\veps}{\varepsilon}
\newcommand{\vphi}{\varphi}
\newcommand{\lbar}{\overline \lambda}

\newcommand{\AppRef}{Appendix~\ref}  

\begin{document}

\title{A novel direct Helmholtz solver in inhomogeneous media based on the
       operator Fourier transform functional calculus}

\author{Max Cubillos and Edwin Jimenez}
\date{}
\maketitle \large

\begin{abstract}
This article presents novel numerical algorithms based on pseudodifferential
operators for fast, direct, solution of the Helmholtz equation in one-, two-
and three-dimensional inhomogeneous unbounded media. The proposed approach
relies on an Operator Fourier Transform (OFT) representation of
pseudodifferential operators ($\PDO$) which frame the problem of computing the
inverse Helmholtz operator, with a spatially-dependent wave speed, in terms of
two sequential applications of an inverse square root pseudodifferential
operator.  The OFT representation of the action of the inverse square root
pseudodifferential operator, in turn, can be effected as a superposition of
solutions of a pseudo-temporal initial-boundary-value problem for a paraxial
equation.  The OFT framework offers several advantages over traditional direct
and iterative approaches for the solution of the Helmholtz equation.  The
operator integral transform is amenable to standard quadrature methods and the
required pseudo-temporal paraxial equation solutions can be obtained using any
suitable numerical method. A specialized quadrature is derived to evaluate the
OFT efficiently and an alternating direction implicit method, used in
conjunction with standard finite differences, is used to solve the requisite
component paraxial equation problems. Numerical studies, in one, two, and three
spatial dimensions, are presented to confirm the expected OFT-based Helmholtz
solver convergence rate. In addition, the efficiency and versatility of our
proposed approach is demonstrated by tackling nontrivial wave propagation
problems, including two-dimensional plane wave scattering from a geometrically
complex inhomogeneity, three-dimensional scattering from turbulent channel flow
and plane wave transmission through a spherically-symmetric gradient-index
Luneburg lens. All computations, even three-dimensional problems which involve
solving the Helmholtz equation with more than one billion complex unknowns, are
performed in a single workstation.
\end{abstract}


\section{Introduction}\label{sec:intro}

We present a \emph{scalable direct solver} for the Helmholtz equation in one-,
two-, and three-dimensional heterogeneous unbounded media based on novel
numerical algorithms for the accurate and efficient evaluation of
pseudodifferential operators.  In our proposed approach, the inverse
variable-coefficient Helmholtz operator is formulated in terms of a composition
of inverse square root pseudodifferential operators.  Our numerical algorithms
are based on a functional calculus framework, which we refer to as the Operator
Fourier Transform (OFT), for the representation of pseudodifferential
operators.  In the OFT framework, the problem of inverting the Helmholtz
operator is reduced to two sequential solves of a pseudo-time paraxial (or
Schr\"odinger) equation and the superposition of these solutions in terms of
pseudo-time Fourier integrals. A salient feature of our proposed approach is
that any numerical method for the solution of time-dependent PDEs can be used
to solve the paraxial equations and any quadrature scheme can be used to
approximate the OFT.

The indefinite Helmholtz equation is notoriously challenging to solve
numerically with iterative methods.  Finite difference, finite element, and
spectral element discretizations of the Helmholtz equation lead to indefinite
linear systems for which iterative methods, including classical stationary
methods and preconditioned Krylov subspace methods, either exhibit slow
convergence or are altogether ineffective~\cite{Erlangga2008,ErnstGander2011}.
For high-frequency three-dimensional problems the computational effort can
quickly become prohibitively expensive since the accuracy is proportional to
both the grid size and the
wavenumber~\cite{BaylissEtAl1985,IhlenburgBabuska1995Disp,IhlenburgBabuska1995FEM,Erlangga2008}.
Numerical experiments with optimized Schwarz methods have shown significant
convergence improvement over Schwarz's original domain decomposition methods,
which date back to the 19th century~\cite{Schwarz1870}, and have been
demonstrated in a variety of nontrivial curvilinear domains but a complete
convergence analysis for this approach is available only for the
non-overlapping subdomain case~\cite{GanderZhang2019}.  More recently, the
development of sweeping preconditioners have lead to iterative Helmholtz
solvers that show residual convergence in a relatively small number of
approximately wavenumber-independent
iterations~\cite{EngquistYing2011a,EngquistYing2011b}.  Although these
preconditioners are promising, it is shown in~\cite{EngquistZhao2018} that the
Helmholtz Green's function is not highly separable in the high-frequency
limit---that is, the number of terms needed in a separable approximation grows
superlinearly with increasing wavenumbers. Because the efficiency of sweeping
preconditioners relies on the separability of submatrices in the discrete
system, wave propagation problems in the high-wavenumber regime will continue
to pose a challenge even for these methods.  Other iterative approaches such as
the WaveHoltz Iteration method, which is based on time-domain solutions of the
wave equation, show promising parallel scalability and improved convergence
over direct discretizations of the Helmholtz equation~\cite{AppeloEtAl2020};
however, similar to all iterative approaches, the efficacy of the WaveHoltz
approach depends on designing a suitable preconditioner.

On the other hand, for certain classes of moderately-sized problems, direct
solvers can provide spectrally-accurate solutions even for scattering problems
in variable media~\cite{GillmanEtAl2015}; they also provide a viable
alternative to iterative methods when solutions are sought with multiple
right-hand sides~\cite{WangEtAl2011}.  The principal drawbacks of direct
solvers typically include an expensive factorization or initialization step as
well as demanding memory requirements, particularly in three dimensions where
the associated linear systems may be comprised of a large number of unknowns.
Some progress has been made to ameliorate the cost of direct solvers by
exploiting sparseness and the low-rank structure of off-diagonal blocks, as in
multifrontal solvers~\cite{WangEtAl2011} and hierarchical matrix
techniques~\cite{BanjaiHackbusch2008};  the efficiency of these methods,
unfortunately, is also subject to the same approximate separability limitations
discussed in~\cite{EngquistZhao2018}. Moreover, parallel scalability of direct
methods is often difficult to achieve.

The OFT approach offers several advantages over other existing Helmholtz
solvers.  First, the cost of the proposed Helmholtz solver is essentially the
cost of two paraxial equation solves. Thus, parallel scaling can be achieved
with domain decomposition approaches for time-dependent PDEs even in complex
curvilinear domains using algorithms such as those described
in~\cite{BrunoEtAl2019}. However, the favorable scaling of our method does not
require a high-performance computing platform to solve large problems nor even
a multi-domain paraxial solver.  Indeed, all computations in this paper rely
only on a single-domain paraxial equation solver and are performed on a single
workstation, including 3D problems with over a billion complex unknowns.  In
addition, the OFT approach allows one to tune the cost of the solver to meet
specific accuracy requirements by adjusting the accuracy of the paraxial
solution. The method is fully analyzable, allowing us to derive error bounds
which are confirmed by numerical convergence tests performed with analytical
solutions. For more complex problems where there is no known exact solution, we
use the asymptotic behavior of the paraxial equation solution and the residual
error to determine an appropriate stopping time for the PDE solves.  For
simplicity, our paraxial equation solver uses standard finite differences to
approximate spatial derivatives and an alternating direction implicit method
derived from the backward Euler scheme (BDF1-ADI) for the pseudo-temporal
evolution.  Guided by our asymptotic analysis of the paraxial equation
solution, we show that it is more efficient to take exponentially-larger time
steps than uniform steps as the solution evolves and the BDF1-ADI time marching
method allows for increasing step sizes without concern for numerical
instabilities.  Finally, to our knowledge, this is the only direct Helmholtz
solver with linear memory scaling, since the memory footprint of the paraxial
solver is linear in the number of unknowns.  Furthermore, there is no expensive
factorization or setup step, unlike other direct methods and some
preconditioners for iterative methods, such as sweeping preconditioners.

Although the focus of this contribution is the solution of the Helmholtz
equation in an unbounded inhomogeneous medium, we emphasize that the OFT
framework is general and can be applied to a broad range of pseudodifferential
operator applications.  Indeed, the OFT first arose as a numerical methodology
in the context of large-scale simulations of high-frequency electromagnetic
propagation~\cite{CubillosJimenez2024}. The time-harmonic Maxwell's equations
can be formally factored into one-way wave equations using pseudodifferential
operators.  In previous work, these operators have been approximated using
rational expansions based on Pad\'e approximation~\cite{KeefeEtAl2018} as well
as AAA-Lawson and Cauchy integral formulations~\cite{KeefeEtAl2024}. The
rational expansion, in turn, leads to a set of large linear systems that must
be solved to evaluate the pseudodifferential operators, which are challenging
to parallelize.  On the other hand, since the OFT relies only on the solution
of pseudo-temporal paraxial equation problems and the evaluation of
Fourier-type integrals, both of which can be tackled with well-established
accurate and efficient numerical methods, the OFT approach naturally leads to
scalable algorithms for the evaluation of pseudodifferential operators.

This paper is organized as follows: Section~\ref{sec:Problem} describes the
setting of the Helmholtz equation considered here. After a brief discussion of
pseudodifferential operators, Section~\ref{sec:OFT} describes the OFT
functional calculus. Section~\ref{sec:OFTSqRtHelm} shows how we apply the OFT
to the Helmholtz problem and describes a simple numerical implementation.
Section~\ref{sec:Error} presents an error analysis, including error estimates,
for the numerical OFT approach.  Section~\ref{sec:NumRes} confirms the accuracy
of the method with convergence analyses in one-, two-, and three-dimensions, as
well as demonstrating the power of the OFT methodology for more complex
wave-scattering problems. Finally, Section~\ref{sec:conclusion} offers some
concluding remarks.

\section{Problem description} \label{sec:Problem}

We consider the Helmholtz equation in an unbounded heterogeneous medium,
\begin{equation}\label{eq:Helmholtz}
  \kappa^2 \refco(\bx) v(\bx) + \Delta v(\bx) = 0, \qquad \bx \in \Dom,
\end{equation}
where $\refco(\bx) > 0$ is a spatially-dependent refraction coefficient,
$\kappa$ is the wavenumber, $\bx = (x_1,\dotsc,x_d) \in \R^d$, for $d = 1,2$ or
$3$, and the Laplacian is $\Delta = \del^2_{x_1} + \dotsb + \del^2_{x_d}$,
where $\del^2_{x_j}$ denotes the second-order derivative with respect to $x_j$.
The Helmholtz equation is of fundamental importance in physics, arising in
models of time-harmonic wave propagation in many fields.  We point out two
important ones.

\paragraph{Electromagnetic waves}

Maxwell's equations for the electric field and displacement ($\bE$ and $\bD$,
respectively) and magnetic field and flux density ($\bH$ and $\bB$,
respectively) in an inhomogeneous medium with no net charge or current sources
are
\begin{align*}
 \nabla \cdot \bD = 0, &\qquad \nabla \times \bE = -\frac{\partial \bB}{\partial t}, \\ 
 \nabla \cdot \bB = 0, &\qquad \nabla \times \bH = \frac{\partial \bD}{\partial t}.
\end{align*}
In a linear medium, we have the constitutive relations $\bD = \veps(\bx) \bE$
and $\bB = \mu(\bx) \bH$, where $\veps(\bx)$ and $\mu(\bx)$ are the
spatially-varying permittivity and permeability of the medium. It is often the
case that the permeability is approximately constant ($\mu(\bx) = \mu_0 =$
const.). If the fields are also assumed to be time-harmonic---that is,
$\bE(\bx,t) = \bEh(\bx) e^{-\ii \omega t}$, where $\omega$ is the frequency of
the radiation, and similarly for the other fields---then Maxwell's equations
can be reduced to the vector Helmholtz equation for the electric field,
\begin{equation*}
  \Delta \bEh - \nabla(\nabla \cdot \bEh) + \frac{\omega^2}{c^2(\bx)} \bEh = 0,
\end{equation*}
where $c(\bx) = 1/\sqrt{\veps(\bx)\mu_0}$ is the spatially-varying speed of
light. In many applications, the electric field is taken to be approximately
divergence free ($\nabla \cdot \bEh \approx 0$). Under this assumption, the
equations are uncoupled and each component of the electric field satisfies the
scalar Helmholtz equation~\eqref{eq:Helmholtz} with $\refco(\bx) =
c_0^2/c^2(\bx)$ and $\kappa = \omega/c_0$, where $c_0$ is a constant reference
speed of light.

\paragraph{Acoustic waves}

The equations for conservation of mass and momentum of a compressible inviscid
fluid are
\begin{align*}
  \frac{\partial \rho}{\partial t} + \nabla \cdot (\rho \bu) &= 0, \\
  \rho \frac{\partial \bu}{\partial t} + \rho \bu \cdot \nabla \bu + \nabla p &= 0,
\end{align*}
where $\rho$, $p$, and $\bu$ are the density, pressure, and velocity of the
fluid, respectively. We linearize the density and velocity around a
(spatially-varying) background field: $\rho = \rho_0(\bx) + \wt \rho$, $\bu =
\bu_0(\bx) + \wt \bu$. Assuming an equation of state of the form $p = p(\rho)$,
we can substitute the time derivative of $\rho$ with that of $p$ by $\partial_t
p \approx \frac{dp}{d\rho} \big|_{\rho_0} \partial_t \wt \rho$. Setting the
background velocity to be zero ($\bu_0(\bx) = 0$), multiplying the mass
conservation equation by $\frac{dp}{d\rho} \big|_{\rho_0}$, taking a second
time derivative, and substituting the momentum conservation equation we have
the wave equation
\begin{equation*}
  \frac{\partial^2 p}{\partial t^2} - \frac{dp}{d\rho} \bigg|_{\rho_0} \Delta p = 0.
\end{equation*}
Assuming a time-harmonic solution $p(\bx,t) = \widehat p(\bx) e^{-\ii \omega
t}$, it follows that $\widehat p(\bx)$ satisfies the Helmholtz
equation~\eqref{eq:Helmholtz} with $\kappa = \omega / c_0$ and $\refco(\bx) =
c_0^2 / c^2(\bx)$, where $\sqrt{\frac{dp}{d\rho} \big|_{\rho_0}} = c(\bx)$ is
the acoustic speed of sound and $c_0$ is a reference sound speed.

\medskip

From the Helmholtz equation, we pose a scattering problem writing the total
field as $\vTot(\bx) = \vInc(\bx) + \vSca(\bx)$, where $\vInc(\bx)$ is an
incident field which we assume satisfies the Helmholtz equation in a
homogeneous medium (i.e., with $\refco(\bx) \equiv 1$), and $\vSca(\bx)$ is the
scattered field.  Substituting this form of the total field into the Helmholtz
equation and dividing by $\kappa^2$, we obtain an equation for the scattered
field
\begin{equation} \label{eq:HelmScat}
\begin{cases}
  \displaystyle \Helmu \vSca(\bx) = g(\bx), & \bx \in \Dom, \\[2ex]
  \displaystyle \vSca(\bx) 
  + \frac{\ii}{\kappa}\frac{\del \vSca(\bx)}{\del n} = 0 , & \bx \in \Bdy,
\end{cases}
\end{equation}
where the source term is given by
\begin{equation}\label{eq:HelmScatSrc}
  g(\bx) \coloneqq -(\refco(\bx) - 1) \vInc(\bx).
\end{equation}
Since in this paper we are primarily concerned with scattering in open domains,
we impose (first-order) non-reflecting boundary conditions at all points of the
domain boundary $\Bdy$. We note that other forms of transparent boundary
conditions, such as perfectly matched layers, could also be imposed with the
method presented in this paper, but, for simplicity, we do not consider them
here.

\section{Pseudodifferential operators and the Operator Fourier Transform} \label{sec:OFT}

Pseudodifferential operators ($\PDO$) are most commonly defined over symbol
classes, which are vector spaces of smooth functions of two vector arguments
and whose mixed derivatives satisfy a boundedness requirement in terms of one
of the vector arguments~\cite{Abels2011,Wong2014}.  In this approach, if
$\sigma(\bx,\bxi)$ is a symbol, then its associated pseudodifferential operator
$\Psi_{\sigma}$ is defined by
\begin{equation}\label{eq:SymbolPDO}
  (\Psi_{\sigma} u)(\bx) \coloneqq \frac{1}{(2\pi)^d}\int_{\R^d} 
    e^{\ii \bx \cdot \bxi} \sigma(\bx,\bxi) \widehat{u}(\bxi) d\bxi,
\end{equation}
where $u:\R^{d}\to \C$ is a smooth function and $\widehat{u}$ denotes its
Fourier transform.  Other pseudodifferential representations also exist, such
as those based on Cauchy's integral theorem and the spectral theorem for normal
operators~\cite{NazaikinskiiEtAl2011}. 

In this work we consider a functional calculus framework for
pseudodifferential operators where a $\PDO$ is expressed in terms of a function
$f$ of an operator argument $A$ so that its application to a function $g$ has
the representation
\begin{equation} \label{eq:OFT}
  [f(A)g](\bx) = \frac 1{\sqrt{2\pi}} \intinf \widehat f(\tau) 
                 e^{\ii \tau A}g(\bx) \,d\tau.
\end{equation}
Following standard convention, $f$ denotes both a function of an operator
argument and a scalar argument and $\widehat f$ is the Fourier transform of $f$
(regarded here as a function of a scalar, $f(x)$).  The $\PDO$ acts on a smooth
function $g$ of a spatial vector variable and the argument $A$ can belong to a
very general operator algebra but for our purposes we take it to be a spatial
derivative operator.  The definition~\eqref{eq:OFT} appears in,
e.g.,~\cite{Taylor1981,Taylor1991,NazaikinskiiEtAl2011}, with various
conditions of validity on the classes of functions $f$ and operators $A$ for
which~\eqref{eq:OFT} is well-defined. There does not seem to be a standard name
associated with this definition, and so henceforth we will refer to it as the
Operator Fourier Transform, or OFT for short (although technically it uses the
inverse Fourier transform as its basis).

The term $u(\bx,\tau) \coloneqq e^{\ii\tau A}g(\bx)$ in~\eqref{eq:OFT}, the
action of the operator $e^{\ii \tau A}$ on $g$, is interpreted as the solution
at a pseudo-time $t=\tau$ of the initial-value problem (IVP)
\begin{equation} \label{eq:IVP}
\begin{cases}
  u_t(\bx,t) = \ii Au(\bx,t),  & (\bx,t) \in \R^d \times (-\infty,\infty), \\
  u(\bx,0) = g(\bx), & \bx \in \R^d.
\end{cases}
\end{equation}
Since our ultimate goal is to solve the time-independent
equation~\eqref{eq:HelmScat} using an inverse (pseudodifferential) Helmholtz
operator, we regard the solution to~\eqref{eq:IVP} as a pseudo-temporal
initial-value problem which serves only as an auxiliary step towards the
evaluation of~\eqref{eq:OFT}.

The OFT overcomes several numerical difficulties associated with the evaluation
of Fourier transform-based pseudodifferential operators of the form defined in
equation~\eqref{eq:SymbolPDO}.  First, in the case of the OFT, singularities in
the function $f$ (which manifest themselves in the Fourier transform $\widehat
f(\tau)$) can be handled using standard integration techniques, such as a
change of variables, specialized quadrature rules, asymptotics, etc. Second,
computing a pseudodifferential operator defined via the OFT is reduced to
solving~\eqref{eq:IVP} and integrating the $\widehat f(\tau)$-weighted solution
in pseudo-time. The IVP~\eqref{eq:IVP} can be solved using any numerical
method, including finite differences, finite element methods, spectral methods,
etc.; the pseudo-time evolution can be effected using explicit time-marching
methods or fast implicit schemes such as alternating direction implicit
methods, or even hybrid implicit-explicit algorithms~\cite{BrunoEtAl2019}.

\subsection{Using the OFT to solve differential equations} \label{sec:OFTExamples}
Before we proceed to the development of our OFT-based inverse Helmholtz
operator, we demonstrate how the OFT framework can be used to solve two simple
ODEs. For the reader's convenience, two Matlab programs \texttt{oft\_exa1.m} and
\texttt{oft\_exa2.m} that implement both examples are included as supplementary
materials.
\begin{exmp}
Consider the following boundary-value problem (BVP) over $[0,1]$,
\begin{align}\label{eq:ODE1}
\begin{cases}
  v(x) - \ii v''(x) = g(x) \coloneqq (1 + \ii \pi^2) \sin(\pi x), & x \in (0,1), \\
  v(x) = 0,   & x = 0,1,
\end{cases}
\end{align}
whose solution is $v(x) = \sin(\pi x)$. Letting  $f(y) = (1-\ii y)^{-1}$, we can
write the solution formally as
\begin{equation}\label{eq:fODE1}
  v(x) = [f\big( \partial_x^2 \big) g](x).
\end{equation}
The Fourier transform of $f$ is given by
\begin{equation}
  \widehat f(\tau) = \sqrt{2\pi} e^{-\tau} H(\tau),
\end{equation}
where $H(\tau)$ is the Heaviside function. Using the OFT definition in
equation~\eqref{eq:OFT} with $A = \partial_x^2$,~\eqref{eq:fODE1} becomes
\begin{equation} \label{eq:ExmpParInt}
  v(x) = \intzeroinf e^{-\tau} e^{\ii \tau \partial_x^2} g(x) \,d\tau. 
\end{equation}
Let $u(x,\tau) = e^{\ii \tau\partial_x^2} g(x)$. Then $u$ is the solution at
time $t = \tau$ of the initial-boundary-value problem (IBVP)
\begin{equation} \label{eq:ExmpParEq}
  \begin{cases}
    \partial_t u(x,t) = \ii \partial_x^2 u(x,t), & (x,t) \in (0,1) \times (0,\infty), \\
    u(0,t) = u(1,t) = 0, & t \in (0,\infty), \\
    u(x,0) = g(x), & x \in [0,1].
  \end{cases}
\end{equation}
The solution of the IBVP~\eqref{eq:ExmpParEq} is simply
\begin{equation}
  u(x,t) = (1+\ii \pi^2) e^{-\ii \pi^2 t} \sin(\pi x).
\end{equation}
The OFT leads to the analytical solution $v(x)$ of the ODE~\eqref{eq:ODE1} in
the form
\begin{align}\label{eq:OFTODE1}
  v(x) &= (1+\ii \pi^2) \intzeroinf e^{-\tau} e^{-\ii \pi^2 \tau} \sin(\pi x) \,d\tau.
\end{align}

\medskip

In the general case, neither the IBVP~\eqref{eq:ExmpParEq} nor the
integral~\eqref{eq:OFTODE1} can be evaluated in closed form.  However, both can
be approximated to any degree of accuracy using appropriate numerical methods.
To demonstrate, we solve~\eqref{eq:ExmpParEq} numerically using a centered
finite difference scheme in space and time and we evaluate~\eqref{eq:OFTODE1}
using the composite trapezoidal rule. Given a positive integer $J$, define the
spatial and temporal discretization points, respectively, as $x_j = j\Delta x$
and $t_n = n\Delta t$ for integers $n\geq 0$ and $0 \leq j \leq J$ and a fixed
spatial grid spacing $\Delta x$ such that $J \Delta x = 1$ and a temporal step
size $\Delta t$ small enough to satisfy the stability condition. Let $v_j
\approx v(x_j)$, $u^n_j \approx u(x_j,t_n)$ be approximations of the unknowns
at the spatio-temporal grid points and let $N$ be a sufficiently large positive
integer such that $e^{-t_N} \approx 0$, where $t_N = N\Delta t$.  A numerical
integration scheme for the OFT integral~\eqref{eq:OFTODE1} is then
\begin{align}
  &v_j = \frac{\Delta t}2 u^0_j + \Delta t \sum_{n=1}^{N-1} e^{-n\Delta t} u^n_j, 
  \qquad 0 \leq j \leq J, \\
\intertext{and the approximate solution to the IBVP~\eqref{eq:ExmpParEq} is obtained from}
  &\begin{cases}
    \dfrac{u^{1}_j - u^{0}_j}{\Delta t} 
    = \ii \dfrac{u^0_{j+1} - 2 u^0_j + u^0_{j-1}}{\Delta x^2}, & 1 \leq j \leq J-1, \\[1ex]
    \dfrac{u^{n+1}_j - u^{n-1}_j}{2\Delta t} 
     = \ii \dfrac{u^n_{j+1} - 2 u^n_j + u^n_{j-1}}{\Delta x^2}, & n \geq 1, 
       \, 1 \leq j \leq J-1, \\[1ex]
    u^n_j = 0, & n \geq 1, \, j = 0,J \\
    u^0_j = g(x_j), & 0 \leq j \leq J.
  \end{cases}
\end{align}
Note that we use a forward-time centered-space scheme for the first time step.
Thus, to evaluate the OFT approximation $v_j$, a cumulative sum is computed as
the numerical solution $u^n_j$ of the IBVP~\eqref{eq:ExmpParEq} progresses up
to a final pseudo-time $t_N = N \Delta t$.  
\end{exmp}
\begin{exmp}
The OFT approach can also be applied to problems over unbounded domains.
Consider the following BVP over $\R$:
\begin{align}\label{eq:ODE2}
\begin{cases}
  v(x) - v''(x) = g(x) \coloneqq (3-4x^2)e^{-x^2}, & x \in \R, \\
  v(x) \rightarrow 0,   & x \rightarrow \pm \infty,
\end{cases}
\end{align}
whose solution is $v(x) = e^{-x^2}$. Letting $f(y) = (1+y^2)^{-1}$ we can write
the solution formally as
\begin{equation}
  v(x) = [f\big( \ii \partial_x \big) g](x).
\end{equation}
In this case the Fourier transform of $f$ is
\begin{equation}
  \widehat f(\tau) = \sqrt{\frac{\pi}{2}} e^{-|\tau|}.
\end{equation}
Using the OFT representation~\eqref{eq:OFT} with $A = \ii \partial_x$, we find
that
\begin{align*}
  v(x) = \frac 12 \intinf e^{-|\tau|} e^{-\tau\partial_x} g(x) \,d\tau 
       = \frac 12 \intzeroinf e^{-\tau} \Big( e^{\tau\partial_x} g(x) 
                          + e^{-\tau\partial_x} g(x) \Big) \,d\tau, 
\end{align*}
where we used the transformation $\tau \rightarrow -\tau$ on the integral from
$-\infty$ to zero. Let $u(x,\tau) = e^{\tau\partial_x} g(x)$ and $w(x,\tau) =
e^{-\tau\partial_x} g(x)$. Then these functions are the solutions at time $t =
\tau$ of the initial-value problems
\begin{subequations} \label{eq:AdvEqs}
\begin{align}
  &\begin{cases}
    \partial_t u(x,t) - \partial_x u(x,t) = 0, & (x,t) \in \R \times (0,\infty), \\
    u(x,0) = g(x), & x \in \R,
  \end{cases} \\
  &\begin{cases}
    \partial_t w(x,t) + \partial_x w(x,t) = 0, & (x,t) \in \R \times (0,\infty), \\
    w(x,0) = g(x), & x \in \R,
  \end{cases}
\end{align}
\end{subequations}
whose solutions are
\begin{equation}
  u(x,t) = g(x+t), \quad w(x,t) = g(x-t).
\end{equation}
It follows that the solution $v(x)$ of the original ODE can be written as
\begin{align}\label{eq:OFTODE2}
  v(x) &= \frac 12 \intzeroinf e^{-\tau} \Big( g(x+\tau) + g(x-\tau) \Big) \,d\tau.
\end{align}

\medskip

We now solve this problem numerically using upwind finite-difference schemes
for the PDEs~\eqref{eq:AdvEqs} and the composite trapezoidal rule for the
integral in equation~\eqref{eq:OFTODE2}. Define the spatial and temporal grid
as in the previous example, with the temporal step size $\Delta t$ small enough
to satisfy the stability condition for this scheme. Let $v_j \approx v(x_j)$,
$u^n_j \approx u(x_j,t_n)$, $w^n_j \approx w(x_j,t_n)$ be approximations of the
unknowns at the spatio-temporal grid points and let $N$ and $J$ be sufficiently
large positive integers such that $e^{-t_N} \approx 0$ and $g(x_{\pm J})
\approx 0$.  The numerical integration scheme for the OFT is then
\begin{align}
  &v_j = \frac{\Delta t}4 \big( u^0_j + w^0_j \big) 
       + \frac{\Delta t}2 \sum_{n=1}^{N-1} e^{-n\Delta t} \big( u^n_j + w^n_j \big), \\
\intertext{and the initial-value problems~\eqref{eq:AdvEqs} are solved using}
  &\begin{cases}
    \dfrac{u^{n+1}_j - u^n_j}{\Delta t} - \dfrac{u^{n}_{j+1} - u^n_j}{\Delta x} = 0, & n \geq 1, \, j \leq J-1, \\[1ex]
    u^n_J = 0, & n \geq 1, \\
    u^0_j = g(x_j), & j \leq J,
  \end{cases} \\
  &\begin{cases}
    \dfrac{w^{n+1}_j - w^n_j}{\Delta t} + \dfrac{w^{n}_{j} - w^n_{j-1}}{\Delta x} = 0, & n \geq 1, \, j \geq -J+1, \\[1ex]
    w^n_{-J} = 0, & n \geq 1, \\
    w^0_j = g(x_j), & j \geq -J.
  \end{cases}
\end{align}
\end{exmp}
%

\section{The OFT applied to the inverse square-root Helmholtz operator} \label{sec:OFTSqRtHelm}

We demonstrate the OFT approach for evaluating pseudodifferential operators by
considering $f(A) = 1/\sqrt{A}$ where the operator $A = \refco(\bx) +
\Delta/\kappa^2$ is the form of the Helmholtz operator used
in~\eqref{eq:HelmScat}. Formally, two sequential applications of the $\PDO$
$f(A)$ yield the inverse Helmholtz operator for~\eqref{eq:HelmScat}.  The
Fourier transform of $f(x) = 1/\sqrt{x}$ is 
\begin{equation}\label{eq:fHat}
  \widehat f(\tau) = \sqrt{\frac{-\ii}{\pi}} H(\tau) \frac 1{\sqrt{\tau}}.
\end{equation}
Throughout this article, the square root of a complex number is taken to be the
principal square root.  Substituting~\eqref{eq:fHat} into
equation~\eqref{eq:OFT} we obtain the following form of the inverse square-root
Helmholtz operator
\begin{align}\label{eq:SqRtHelmOFT}
  f(A) = \sqrt{\frac{-\ii}{\pi}} \intzeroinf \frac 1{\sqrt{\tau}} e^{\ii\tau A} \,d\tau 
       = \sqrt{\frac{-\ii}{\pi}} \intzeroinf 
         \frac{e^{\ii\tau}}{\sqrt{\tau}} e^{\ii\tau (A-I)} \,d\tau. 
\end{align}
To evaluate the operator~\eqref{eq:SqRtHelmOFT} applied to a function $g(\bx)$
using the OFT approach, we define $u(\bx,\tau) = e^{\ii\tau (A-I)} g(\bx)$.
Then, to effect the action of the $\PDO$~\eqref{eq:SqRtHelmOFT} on $g(\bx)$
over a domain $\Dom \subset \R^d$ means that we must solve the
initial-boundary-value problem
\begin{equation} \label{eq:ParaxialEqn}
\begin{cases}
  u_t(\bx,t) = \ii \big(\refco(\bx) - 1\big) u(\bx,t) + \dfrac{\ii}{\kappa^2} \Delta u(\bx,t), & (\bx,t) \in \Dom \times (0,\infty), \\[1ex]
  u(\bx,t) + \dfrac{\ii}{\kappa} \dfrac{\partial u(\bx,t)}{\partial n} = 0, & (\bx,t) \in \Bdy \times (0,\infty), \\[1ex]
  u(\bx,0) = g(\bx), & \bx \in \Dom.
\end{cases}
\end{equation}

We can easily verify that $f^2(A)g = (f \circ f)(A)g$ does indeed solve the
Helmholtz equation~\eqref{eq:HelmScat}.  Since the paraxial equation is solved
imposing the same homogeneous boundary conditions as the Helmholtz
problem~\eqref{eq:HelmScat} for all time, any superposition will also satisfy
those boundary conditions. We need only check that $f^2(A)$ is a multiplicative
(composition) inverse of $A$:
\begin{align}
  f^2(A) A &= \frac{1}{2\pi} \dblintinf \widehat f(s) \widehat f(\tau) e^{\ii (s + \tau) A} A\,ds\,d\tau \nonumber \\
    &= \frac{-\ii}{\sqrt{2\pi}} \intinf  \ii A e^{\ii r A} \bigg( \frac 1{\sqrt{2\pi}} \intinf \widehat f(r-s) \widehat f(s)\,ds \bigg) \,dr. \nonumber
\end{align}
By the convolution theorem, the term in parentheses is $\mathcal
F\big[x^{-1}\big](r) = -\ii \sqrt{\pi/2} \sgn(r)$. Therefore, to integrate by
parts we only assume that the norm of the paraxial solution satisfies $\|
e^{\ii t A} \| = o( 1 )$ as $t \rightarrow \infty$, which we will see is indeed
the case for the boundary conditions imposed. In this case we have
\begin{align}
  f^2(A) A &= \frac{-\ii}{\sqrt{2\pi}} \intinf e^{\ii r A} \bigg( \frac 1{\sqrt{2\pi}} \intinf \widehat f'(r-s) \widehat f(s)\,ds \bigg) \,dr. 
\end{align}
The term in parentheses is the Fourier transform of $-\ii x f^2(x) = -\ii$, so
we have
\begin{align}
  f^2(A) A &= \frac{\ii}{\sqrt{2\pi}} \intinf e^{\ii r A} \Big( -\ii \sqrt{2\pi}\,\delta(r) \Big) \,dr = I. \nonumber
\end{align}

A numerical method for evaluating $f(A)$ using the OFT now hinges on the
numerical approximations of the integral in equation~\eqref{eq:SqRtHelmOFT} and
the solution of the PDE~\eqref{eq:ParaxialEqn}. Algorithms to tackle both of
these problems are developed in the following subsections.

\subsection{Numerical computation of the OFT}\label{sec:NumericalOFT}

In this section, we develop a quadrature for the integral
in~\eqref{eq:SqRtHelmOFT} based on the behavior of the low, medium, and high
frequency content of the solution to the paraxial
equation~\eqref{eq:ParaxialEqn} in a slowly-varying medium.

For simplicity, we consider scattering of an incident plane wave $\vInc =
e^{\ii \kappa \bp \cdot \bx}$, where $\bp$ is a unit vector indicating the
propagation direction. We write the wavespeed inhomogeneity as $\refco(\bx) = 1 +
\veps \rho(\bx)$ and consider the asymptotic regime $\veps \ll 1$. The initial
condition of the paraxial equation will be assumed to be of form $g(\bx) =
e^{\ii \kappa \bp \cdot \bx} h(\bx)$. If we write the solution as $u(\bx,t) =
e^{\ii (\kappa \bp \cdot \bx - t)} w(\bx,t)$ and substitute this form into
equation~\eqref{eq:ParaxialEqn}, then $w$ satisfies the equation
\begin{equation} \label{eq:ParaxialW}
\begin{cases}
  w_t + \dfrac 2{\kappa} \bp \cdot \nabla w 
  = \ii\veps\rho\, w + \dfrac{\ii}{\kappa^2} \Delta w, 
  & (\bx,t) \in \R^d \times (0,\infty), \\[1ex]
  w(\bx,0) = h(\bx), & \bx \in \R^d.
\end{cases}
\end{equation}
The left-hand side of the PDE~\eqref{eq:ParaxialW} has an advective term with
velocity $2/\kappa$, and for large $\kappa$ and small $\veps$ this is the
dominant behavior. The solution for just the advective part is $w(\bx,t) \sim
h(\bx - (2/\kappa)t \bp)$, and if the computational domain is of width $L$, this
advective part of the solution will leave the domain at time $t \sim \kappa
L/2$. We will call the $t \lesssim \kappa L/2$ part of the solution evolution
the ``advection regime.'' This can also be considered the ``high-frequency
regime,'' since most of the high-frequency content of the solution is leaving
the domain here.

On the other hand, the asymptotic behavior of the solution is described for $t
\gtrsim (\kappa L)^2$, $\kappa L \gg 1$ in~\ref{sec:ExactSoln}; this is the
``asymptotic regime.'' It can also be described as the ``low-frequency
regime,'' since the spatial component of the asymptotic solution is approaching
the lowest eigenmode.

In between the advection and asymptotic regimes---that is, for $\kappa L
\lesssim t \lesssim (\kappa L)^2$---we have the ``transition regime,'' or
``mid-frequency regime.'' Although we do not have an exact description of the
behavior of the solution here, we expect that higher frequencies will continue
to leave the domain until all that remains are the lowest order modes governed
by the asymptotic regime.

\medskip

Based on this heuristic analysis of the paraxial solution, we expect that an
optimal strategy for numerically computing the OFT must be adapted to each
regime. The advection regime will require the highest resolution in both space
and time as the main features of the initial condition leave the domain.
Subsequently, spatial and temporal discretization requirements may gradually be
relaxed until the asymptotic regime dominates, where the lowest resolution is
required.  Hence, for maximum efficiency, both the spatial and temporal grid
resolutions should be coarsened, and thus require less computational work,
after the solution to~\eqref{eq:ParaxialEqn} progresses beyond the advection
regime.

However, to simplify the presentation of the method as well as the
corresponding error analysis in Section~\ref{sec:Error}, we will describe a
more straightforward approach. For the spatial discretization we use
second-order finite differences over a fixed, uniformly spaced grid. The
pseudo-temporal evolution of the paraxial equation is effected with a
first-order alternating direction implicit method based on backward Euler
(BDF1-ADI) together with an exponentially increasing time-step size.  Finally,
the OFT quadrature is based on linear interpolation between time step
intervals. Complete numerical algorithm details are provided in the
following subsections.

\subsection{Temporal and spatial discretization of the paraxial equation}\label{sec:Discretization}

We discretize equation~\eqref{eq:ParaxialEqn} in time using the first-order
BDF1-ADI method~\cite{BrunoEtAl2019}. Let the time steps be given by $t_{n+1} =
t_n + \Delta t_n$. In 3D, the split equation for obtaining $u^{n+1}$ from $u^n$
is given by
\begin{subequations} \label{eq:BDF1Split3D}
\begin{align}
  \big( I + \Delta t_n \A \big) u^{(1)} &= u^n, \\
  \big( I + \Delta t_n \B \big) u^{(2)} &= u^{(1)}, \\
  \big( I + \Delta t_n \CC \big) u^{n+1} &= u^{(2)},
\end{align}
\end{subequations}
where $u^{(1)}$ and $u^{(2)}$ are intermediate unknowns, and the spatial
operators are defined as 
\begin{equation}
  \A = -\ii (\refco -1) I - \ii \kappa^{-2} \partial_x^2, \quad \B = -\ii \kappa^{-2} \partial_y^2, \quad \CC = -\ii \kappa^{-2} \partial_z^2.
\end{equation}
These operators are then approximated by a centered finite difference scheme---e.g.,
\begin{equation}
  \partial_x^2 u^n_{ijk} \approx \delta_x^2 u^n_{ijk} = \frac{u^n_{i+1,jk} - 2 u^n_{ijk} + u^n_{i-1,jk}}{\Delta x^2},
\end{equation}
where $i,j,k$ are the grid indices and $\Delta x$ is the grid spacing in the
$x$ direction. The boundary conditions in~\eqref{eq:HelmScat} are approximated
by three-point second-order one-sided differences, which preserve the spatial
order. The ODEs in equation~\eqref{eq:BDF1Split3D} are then solved line-by-line
in each dimension with a tridiagonal linear system solver.

\subsection{Exponential time-stepping}\label{sec:TimeStepping}

The exponential time-stepping scheme we use relies on temporal nodes of the
form
\begin{equation} \label{eq:ExpTimeSteps}
  t_n = a (b^n - 1), \qquad n \geq 0.
\end{equation}
It is useful to choose the constants $a$ and $b$ such that the initial
time-step size is $\Delta t_0$ and the time-step size at some later time $T$ is
$\Delta t_T$. In other words, given prescribed step sizes $\Delta t_0$ and
$\Delta t_T$, we satisfy the equations
\begin{equation}
  \begin{cases}
    t_1 = a (b - 1) = \Delta t_0, \\
    t_N = a (b^N - 1) = T, \\
    \Delta t_N = t_{N+1} - t_N = \Delta t_T,
  \end{cases}
\end{equation}
by setting 
\begin{equation} \label{eq:ExpTimeStepsPrms}
  a = \frac{T}{R}, \quad b = 1 + R \frac{\Delta t_0}{T}, \quad R = \frac{\Delta t_T}{\Delta t_0} - 1.
\end{equation}
%

\subsection{OFT quadrature}\label{sec:OFTQuad}

We first define the linear interpolant between two time steps of a sequence
$\{u^m\}$ as
\begin{equation}
 \I_n u^m (t) = \frac{t_{n+1}-t}{\Delta t_n} u^n + \frac{t-t_n}{\Delta t_n} u^{n+1}.
\end{equation}
The OFT integral is then approximated by
\begin{align}
\sqrt{\frac{-\ii}{\pi}} \intzeroinf \frac{e^{\ii\tau}}{\sqrt{\tau}}\, u(\tau) \,d\tau &= \sum_{n=0}^{\infty} \sqrt{\frac{-\ii}{\pi}} \int_{t_n}^{t_{n+1}} \frac{e^{\ii\tau}}{\sqrt{\tau}}\, u(\tau) \,d\tau, \nonumber \\
  &\approx \sum_{n=0}^N \sqrt{\frac{-\ii}{\pi}} \int_{t_n}^{t_{n+1}} 
             \frac{e^{\ii\tau}}{\sqrt{\tau}}\, \I_nu^m (\tau) \,d\tau
\end{align}
The integrals of each component of the interpolant can be written in closed
form:
\begin{align}
  w_1(a,b) &= \sqrt{\frac{-\ii}{\pi}} \int_a^b 
              \frac{e^{\ii\tau}}{\sqrt{\tau}} \frac{b-\tau}{b-a} \nonumber \\
           &= \frac{1+\ii}{\sqrt{2\pi}(b-a)} 
              \Big[ \sqrt{b} e^{\ii b} - \sqrt{a} e^{\ii a} \nonumber \\
           &\qquad + (1+2\ii b) \big( C(\sqrt{a}) - C(\sqrt{b}) 
                                    + \ii S(\sqrt{a}) - \ii S(\sqrt{b}) \big) \Big], \\
  w_2(a,b) &= \sqrt{\frac{-\ii}{\pi}} \int_a^b \frac{e^{\ii\tau}}{\sqrt{\tau}} 
              \frac{\tau-a}{b-a} \nonumber \\
           &= \frac{1+\ii}{\sqrt{2\pi}(b-a)} 
              \Big[ \sqrt{a} e^{\ii a} - \sqrt{b} e^{\ii b} \nonumber \\
           &\qquad + (1+2\ii a) \big( C(\sqrt{b}) - C(\sqrt{a}) 
                                    + \ii S(\sqrt{b}) - \ii S(\sqrt{a}) \big) \Big],
\end{align}
where $C(x)$ and $S(x)$ are the Fresnel cosine and sine integral special functions,
\begin{equation}
  C(x) = \int_0^x \cos(t^2)\,dt, \quad S(x) = \int_0^x \sin(t^2)\,dt.
\end{equation}
Using these formulas, we can write a composite quadrature rule,
\begin{align}
\sqrt{\frac{-\ii}{\pi}} \intzeroinf \frac{e^{\ii\tau}}{\sqrt{\tau}}\, u(\tau) \,d\tau &\approx \sum_{n=0}^{N} \omega_n u^n
\end{align}
where the weights are given by
\begin{equation} \label{eq:OFTWeights}
  \omega_n =
  \begin{cases}
    w_1(0,t_1), & n = 0, \\
    w_2(t_{n-1},t_n) + w_1(t_n,t_{n+1}), & 1 \leq n \leq N-1, \\
    w_2(t_{N-1},t_N), & n = N
  \end{cases}
\end{equation}
%

\section{Error analysis}\label{sec:Error}

In this section we derive an error estimate for the numerical OFT presented in
the previous section in the constant coefficient case in 1D. To simplify the
analysis, we do not include the error of the finite-difference spatial
discretization since this is not intrinsic to the OFT framework. Instead, we
assume a domain $\Dom$ and a linear space of functions $V(\Dom)$ with norm
$\|\cdot\|$ such that the operator $A$ is bounded, there is $\sigma > 0$ such
that the solution operator $S(t) = e^{\ii t (A-I)}$ satisfies
\begin{equation} \label{eq:SolnIneq}
  \| S(t) \| \leq e^{-\sigma t}, \quad t > 0,
\end{equation}
and the backward Euler operator $B(\Delta t) = \big( I - \ii \Delta t (A-I)
\big)^{-1}$ is both consistent and strictly unconditionally stable (and
therefore convergent), satisfying the inequality
\begin{equation} \label{eq:BackEulerIneq}
  \| B(\Delta t) \| \leq \frac 1{1 + \rho(C) \Delta t}, \quad 0 < \Delta t < C,
\end{equation}
for any $C > 0$, where $\rho(C) > 0$ is a constant that depends on $C$.
Consistency implies $\rho \rightarrow \sigma$ as $C \rightarrow 0$. The
approximate solution operator $\wt S_n$ is then defined by the
recurrence,
\begin{equation}
  \wt S_{n+1} = B_n \wt S_n, \quad \wt S_0 = I, \quad B_n = B(\Delta t_n).
\end{equation}
The above assumptions are satisfied, for example, by taking $\Dom$ to be an
interval of length $L$, $K$ a positive integer, $V_K(\Dom)$ the space of all
finite series of the form
\begin{equation}
  v(x) = \sum_{k=1}^K c_k \vphi_k(x),
\end{equation}
where $c_k \in \C$ and $\vphi_k(x)$ are the eigenfunctions given in the
appendix, and using the norm
\begin{equation}
  \|v(x)\| = \bigg( \sum_{k=1}^K |c_k|^2 \bigg)^{1/2}.
\end{equation}
(This is not the same as the norm induced by the $L^2$ inner product on
$\Dom$.) In the appendix it is shown that all eigenvalues have negative
imaginary part so that the inequalities~\eqref{eq:SolnIneq}
and~\eqref{eq:BackEulerIneq} hold.

With these definitions, we have the following estimate.

\begin{thm} \label{thm:Error}
Let $f(A)$ be given by equation~\eqref{eq:SqRtHelmOFT} and let
\begin{equation}
f^N(A) = \sum_{n=0}^N \omega_n \wt S_n,
\end{equation}
where the weights are given by equation~\eqref{eq:OFTWeights} and the sequence
of time steps is given by equations~\eqref{eq:ExpTimeSteps}
and~\eqref{eq:ExpTimeStepsPrms}. Then the error is bounded by
\begin{equation}
  \| f(A) - f^N(A) \| \leq E_1 + E_2 + E_3,
\end{equation}
where, fixing $\sigma$ and $R$ and letting $T \rightarrow \infty$ and $\Delta
t_0 \rightarrow 0$, the error terms are given asymptotically by
\begin{align}
  E_1 \lesssim \frac{\|A-I\|^2}{12\sqrt{\sigma}} \Delta t_0^2, \quad E_2 \lesssim \frac{\|A-I\|^2}{8\sigma^{3/2}} \Delta t_0, \quad E_3 \lesssim \frac{e^{-\sigma T}}{\sigma \sqrt{\pi T}}.
\end{align}
$E_1$ is the error incurred by the piecewise linear quadrature, $E_2$ is
the error due to backward Euler, and $E_3$ is the error of truncating the OFT
integral at $t = t_N = T$.
\end{thm}

The proof is given in \AppRef{sec:ErrorProof}.

\section{Numerical results}\label{sec:NumRes}

In this section we present convergence studies in one, two and three spatial
dimensions that verify the stability, accuracy, and efficiency of our OFT-based
Helmholtz equation solvers. We also demonstrate our numerical algorithms with
plane wave scattering and transmission problems through complex two- and
three-dimensional inhomogeneous media.  We emphasize that all numerical results
presented in this section, even three-dimensional examples that require
solutions with more than one billion unknowns, were performed in a single AMD
EPYC 7543P $32$-core workstation. (Images were generated using the
visualization software VisIt~\cite{VisIt2012}.)

\subsection{Convergence studies for the $\isrHel$ and $\iHel$ operators}\label{sec:ConvRes}

To verify the accuracy of our proposed algorithms for the evaluation of both
the $\isrHel$ and $\iHel$ pseudodifferential operators, we solve the following
boundary-value problems (BVP)
\begin{subequations} \label{eq:TwoPassHel}
\begin{align} 
&\begin{cases}
  \srHel v_1(\bx) = g(\bx), & \bx \in \Dom, \\
  \displaystyle v_1(\bx) 
  + \frac{i}{\kappa}\frac{\del v_1(\bx)}{\del n} = 0 , & \bx \in \Bdy,
\end{cases}\label{eq:TwoPassHel1} \\
&\begin{cases}
  \srHel v_2(\bx) = v_1(\bx), & \bx \in \Dom, \\
  \displaystyle v_2(\bx) 
  + \frac{i}{\kappa}\frac{\del v_2(\bx)}{\del n} = 0 , & \bx \in \Bdy,
\end{cases}\label{eq:TwoPassHel2}
\end{align}
\end{subequations}
where $\bx = (x_1,\dotsc,x_d)$ and $\Dom = \unitInt^d$, for $d = 1,2,3$. In
this case we set the refraction coefficient $\refco(\bx) \equiv 1$ and the
wavenumber $\kappa = 10$.  In all cases, the source term $g(\bx) =
e^{-a_0|\bx|^2 + \ii\kappa x_1}$, with parameter $a_0 = 10$. 

Each paraxial equation initial-boundary-value problem (IBVP) associated
with~\eqref{eq:TwoPassHel} can be written succinctly as
\begin{equation}\label{eq:ParForTwo}
\begin{cases}
  \displaystyle u_t(\bx,t) 
  = \frac{i}{\kappa^2}\Delta u(\bx,t), & (\bx,t) \in \Dom \times (0,\Tinf], \\[1ex]
  \displaystyle u(\bx,t) 
  + \frac{i}{\kappa}\frac{\del u(\bx,t)}{\del n} = 0 , & (\bx,t) \in \Bdy \times (0,\Tinf], \\[1ex]
  u(\bx,0) = h( \bx ), & \bx \in \Dom,
\end{cases}
\end{equation}
where only the initial condition $h(\bx)$ differs in each case.
For~\eqref{eq:TwoPassHel1} $h(\bx) = g(\bx)$, while
problem~\eqref{eq:TwoPassHel2} requires $h(\bx) = v_1(\bx)$. Note that both
problems are solved up to the same final pseudo-time $t=\Tinf$.

Using the non-reflecting boundary conditions eigenfunction basis detailed
in~\ref{sec:ExactSoln}, we can write the exact solutions
to~\eqref{eq:TwoPassHel} for the given $g(\bx)$. In three spatial dimensions,
for example, the solutions to~\eqref{eq:TwoPassHel} are
\begin{subequations}\label{eq:ExaSolNRBC}
\begin{align}
  v_1( x_1, x_2, x_3 ) &= \sum_{\ell=1}^{\infty}\sum_{m=1}^{\infty}\sum_{n=1}^{\infty}
                \frac{ c_{\ell m n} }
                { \sqrt{1 - (\lambda_{1,\ell}^2 + \lambda_{2,m}^2 + \lambda_{3,n}^2)/\kappa^2} }
                \varphi_{\ell}( x_1 ) \varphi_{m}( x_2 ) \varphi_{n}( x_3 ) \label{eq:ExaSolNRBC1}\\
  v_2( x_1, x_2, x_3 ) &= \sum_{\ell=1}^{\infty}\sum_{m=1}^{\infty}\sum_{n=1}^{\infty}
                \frac{ c_{\ell m n} }
                { 1 - (\lambda_{1,\ell}^2 + \lambda_{2,m}^2 + \lambda_{3,n}^2)/\kappa^2 }
                \varphi_{\ell}( x_1 ) \varphi_{m}( x_2 ) \varphi_{n}( x_3 ) \label{eq:ExaSolNRBC2}
\end{align}
\end{subequations}
The eigenfunctions and eigenvalues along $x_1$ are denoted by
$\varphi_{\ell}(x_1)$ and $\lambda_{1,\ell}$, respectively, and similarly for
the other dimensions. The coefficients $c_{\ell m n}$ are derived using the
source function $g(\bx)$ (see \AppRef{sec:ExactSoln} for additional
details). The exact solutions $v_1$ and $v_2$ in one and two dimensions are
obtained analogously.

Tables~\ref{tab:ConvStudyPass1} and~\ref{tab:ConvStudyPass2} present relative
errors for numerical solutions to~\eqref{eq:TwoPassHel}, for $\kappa=10$,
obtained with various spatial and temporal resolutions.  Across each row in the
table, we use the same exponential pseudo-temporal step sequence obtained with
$\Delta t_0$ and $\Delta t_T = 10\Delta t_0$, $T = \kappa L$ where $L = 2$ is
the length of the domain, and a uniform spatial size $\Delta x \,\, (=\Delta
x_1=\Delta x_2=\Delta x_3)$ for each one-, two- and three-dimensional BVP. The
number of pseudo-time steps is denoted by $\Ntau$ and the number of points
along each spatial dimension by $N_{x}$. The 2D and 3D problems use the same
spatial discretization ($\Delta x$) of $[-1,1]^d$ along all dimensions. A
single spatial grid point is written as $\bx_j$, where $j=(j_1,\dotsc,j_d)$ and
the value of a function at a grid point is denoted by $v_j$.  After setting
$\Delta t_0 \in \{10^{-1},10^{-2},\dotsc,10^{-5}\}$, we form the OFT quadrature
nodes from equation~\eqref{eq:ExpTimeSteps} and the corresponding OFT
quadrature weights from~\eqref{eq:OFTWeights}.  (Note that $T_f =
t_{N_{\tau}}.$) Denoting an exact solution by $v^{exa}$ and an approximate
solution by $v^{app}$, where $v \in \{v_1,v_2\}$, we compute the relative error
over all spatial grid points $j$ as
\begin{equation}
  \relErr = \frac{\max_{j}{|v_j^{exa} - v_j^{app}|}}{\max_{j}{|v_j^{exa}|}}
\end{equation}
We write $\relErr^{1D}, \relErr^{2D}$, and $\relErr^{3D}$ to distinguish errors
obtained in 1D, 2D, and 3D. Table~\ref{tab:ConvStudyPass2} also includes the
relative residual error $\res$, defined as
\begin{equation}
  \res = \frac{\max_{j}{\Big| \big( (I + \Delta/\kappa^2)v^{app} - g \big)_j \Big|}}{\max_{j}{|g_j|}},
\end{equation}
and Table~\ref{tab:ConvStudyPass1} also includes values for $\ub$, which are
defined as
\begin{equation}
  \ub = \frac{\max_j |v_j^{app}|}{\sigma \sqrt{T_f}},
\end{equation}
with $\sigma = (\kappa L)^{-1}$, which is an estimate for the OFT truncation
error $E_3$ given in Theorem~\ref{thm:Error} and can be used heuristically as a
stopping criteria for first application of the inverse square-root Helmholtz
operator. The table suggests that for this problem, in order to achieve an
error tolerance of $\veps_{tol}$, the paraxial solve should run until time $t$
such that $\max_j |v^{app}_j| \lesssim 10 \veps_{tol} \sigma \sqrt{t}$.
\begin{table}[h!] 
\centering
\caption{Relative errors in the evaluation of the operator $\isrHel$ 
         applied to the function $g(\bx) = e^{-a_0|\bx|^2 + \ii\kappa x_1}$, 
         where $\kappa = 10$ and $a_0 = 10$. In all cases, the exponential 
         time-stepping parameter $\Delta t_T = 10 \Delta t_0$.}
\begin{tabular}{lcccccccc}\toprule
  &&& \multicolumn{6}{c}{$v_1=\isrHel g$}
  \\ \cmidrule(lr){4-9}
  $\Delta t_0$        &$\Ntau$    &$\Nspc$  &$\relErr^{1D}$      &$\ub^{1D}$          &$\relErr^{2D}$      &$\ub^{2D}$          &$\relErr^{3D}$      &$\ub^{3D}$ \\
  \midrule
  $5.0\cdot 10^{-2}$  &$102$      &$70$     &$1.2\cdot 10^{-1}$  &$4.9\cdot 10^{-1}$  &$7.4\cdot 10^{-2}$  &$2.0\cdot 10^{-1}$  &$4.8\cdot 10^{-2}$  &$8.2\cdot 10^{-2}$ \\
  $5.0\cdot 10^{-3}$  &$1308$     &$200$    &$1.3\cdot 10^{-2}$  &$2.0\cdot 10^{-1}$  &$8.2\cdot 10^{-3}$  &$9.1\cdot 10^{-2}$  &$5.3\cdot 10^{-3}$  &$4.1\cdot 10^{-2}$ \\
  $5.0\cdot 10^{-4}$  &$17810$    &$600$    &$1.8\cdot 10^{-3}$  &$6.0\cdot 10^{-2}$  &$8.8\cdot 10^{-4}$  &$1.7\cdot 10^{-2}$  &$5.4\cdot 10^{-4}$  &$5.1\cdot 10^{-3}$ \\
  $5.0\cdot 10^{-5}$  &$233199$   &$1800$   &$1.8\cdot 10^{-4}$  &$7.1\cdot 10^{-3}$  &$8.7\cdot 10^{-5}$  &$9.7\cdot 10^{-4}$  &---                 &--- \\
  $5.0\cdot 10^{-6}$  &$2617277$  &$5400$   &$1.9\cdot 10^{-5}$  &$8.1\cdot 10^{-4}$  &---                 &---                 &---                 &--- \\
 \bottomrule
\end{tabular}
\label{tab:ConvStudyPass1}
\end{table}
\begin{table}[h!] 
\centering
\caption{Relative errors in the evaluation of the operator $\iHel$ 
         applied to the function $g(\bx) = e^{-a_0|\bx|^2 + \ii\kappa x_1}$, 
         where $\kappa = 10$ and $a_0 = 10$. In all cases, the exponential 
         time-stepping parameter $\Delta t_T = 10 \Delta t_0$.}
\begin{tabular}{lcccccccc}\toprule
  &&& \multicolumn{6}{c}{$v_2=\iHel g$}
  \\ \cmidrule(lr){4-9}
  $\Delta t_0$        &$\Ntau$    &$\Nspc$  &$\relErr^{1D}$      &$\res^{1D}$         &$\relErr^{2D}$      &$\res^{2D}$         &$\relErr^{3D}$      &$\res^{3D}$ \\
  \midrule
  $5.0\cdot 10^{-2}$  &$102$      &$70$     &$2.3\cdot 10^{-1}$  &$1.7\cdot 10^{-1}$  &$1.6\cdot 10^{-1}$  &$1.0\cdot 10^{-1}$  &$1.1\cdot 10^{-1}$  &$8.6\cdot 10^{-2}$ \\
  $5.0\cdot 10^{-3}$  &$1308$     &$200$    &$2.5\cdot 10^{-2}$  &$2.3\cdot 10^{-2}$  &$1.8\cdot 10^{-2}$  &$1.4\cdot 10^{-2}$  &$1.2\cdot 10^{-2}$  &$9.9\cdot 10^{-3}$ \\
  $5.0\cdot 10^{-4}$  &$17810$    &$600$    &$2.5\cdot 10^{-3}$  &$4.3\cdot 10^{-3}$  &$1.8\cdot 10^{-3}$  &$2.0\cdot 10^{-3}$  &$1.3\cdot 10^{-3}$  &$1.0\cdot 10^{-3}$ \\
  $5.0\cdot 10^{-5}$  &$233199$   &$1800$   &$2.5\cdot 10^{-4}$  &$5.0\cdot 10^{-4}$  &$1.9\cdot 10^{-4}$  &$1.2\cdot 10^{-4}$  &---                 &--- \\
  $5.0\cdot 10^{-6}$  &$2617277$  &$5400$   &$2.4\cdot 10^{-5}$  &$5.5\cdot 10^{-5}$  &---                 &---                 &---                 &--- \\
 \bottomrule
\end{tabular}
\label{tab:ConvStudyPass2}
\end{table}

The results reported in Table~\ref{tab:ConvStudyPass1} clearly indicate that each
refinement of $\Delta t$ by a factor of $10$ leads to an order of magnitude
decrease in the relative error in the approximation of the $\PDO$s $\isrHel$
and $\iHel$, which confirm the expected first-order convergence predicted in
Theorem~\ref{thm:Error}. We note that the error decrease is consistent in all
three dimensions and is in fact slightly lower in higher dimensions.  This is
consistent with the physics of wave scattering and the asymptotic analysis
leading to equation~\eqref{eq:UAsym}, which show that waves decay faster in
higher dimensions. Since the numerical approximation to $v_1$ is used as an
initial condition for the second BVP~\eqref{eq:TwoPassHel2}, the evaluation of
$\iHel g$ is subject to additional error. The data shows that, for all spatial
dimensions and all discretization parameters considered, the error associated
with the inverse Helmholtz pseudodifferential operator is in fact nearly twice
that of the $\isrHel g$ computations.

\subsection{Two dimensional plane wave scattering from an inhomogeneous obstacle}\label{sec:AFLogo}
%
\begin{figure}[h!]
  \centering
  \includegraphics[width=\textwidth]{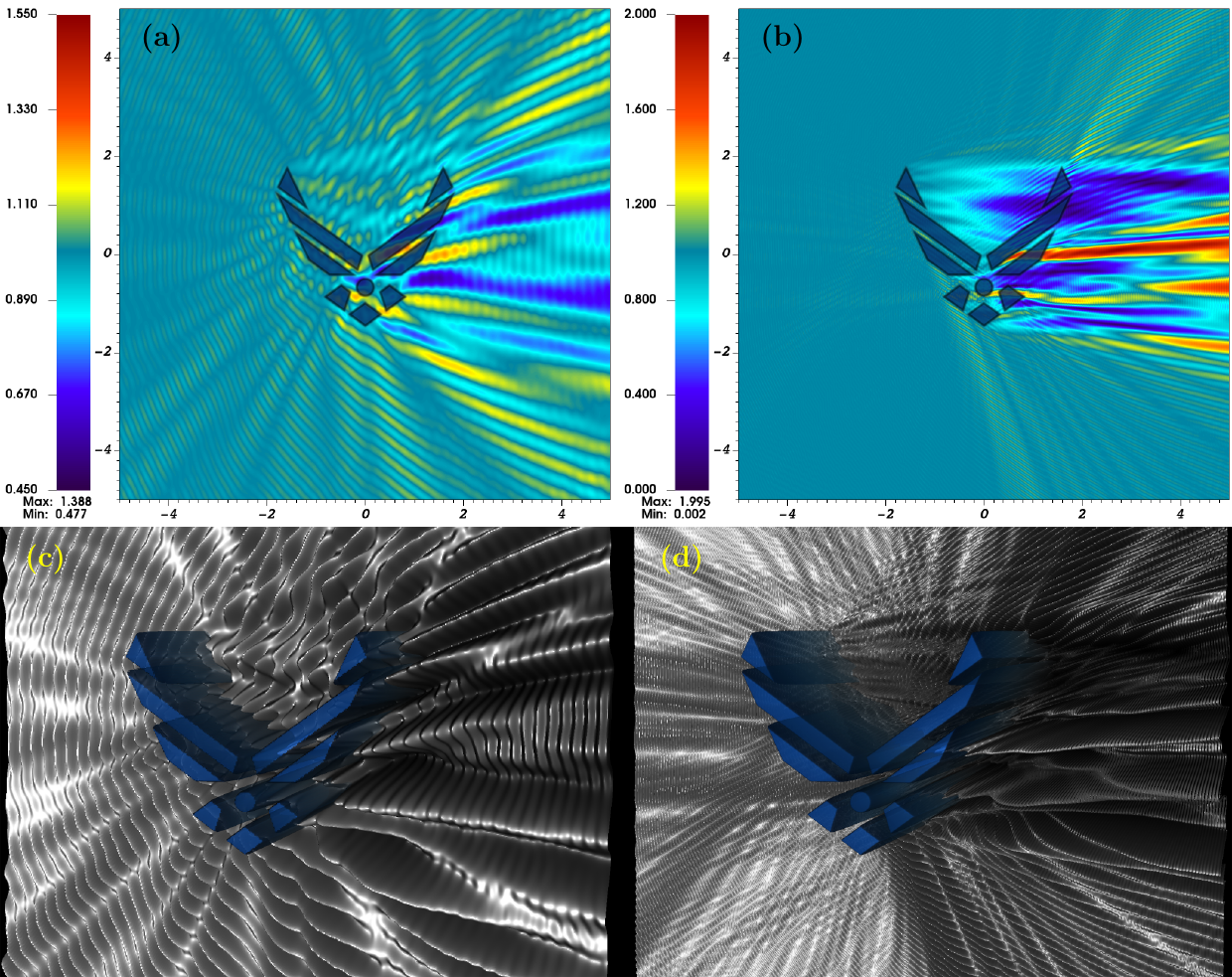}
  \caption{Total field magnitude $|\vTot|=|\vInc+\vSca|$ for plane wave scattering 
           through a variable 
           wave speed inhomogeneity in the shape of the U.S. Air Force logo. 
           Total field magnitude with highlighted logo and surrounding partially 
           reflected/transmitted
           waves for \emph{(a)} 25 wavelengths and \emph{(b)} 100 wavelengths, respectively,
           along each dimension.  Elevated plots of the total field magnitude for \emph{(c)} 25 
           wavelengths and \emph{(d)} 100 wavelengths, which highlight the interaction 
           of the incoming wave with the inhomogeneity.} 
  \label{fig:AFLogo}
\end{figure}
In this section we present computational results for plane wave scattering from
a geometrically complex inhomogeneity contained in a two-dimensional domain
$\Dom = [-5,5]^2$.  We solve the Helmholtz equation~\eqref{eq:HelmScat} for the
scattered field $\vSca(\bx)$ where, in this case, $\bx = (x_1,x_2)$, and the
incident field is $\vInc(\bx) = e^{i\kappa x_1}$. The spatially-dependent
refractive coefficient $\refco(\bx) = 1 + \delta \refco(\bx)$ is in the shape of the
U.S.~Air Force logo (see Figure~\ref{fig:AFLogo}), where $0 \leq \delta
\refco(\bx) \leq 10^{-1}$. The source term $g(\bx) = -(\refco(\bx) - 1) \vInc(\bx).$

Conceptually, the solution to~\eqref{eq:HelmScat} is obtained simply as $\vSca
= \Psi^2 g = (\Psi \circ \Psi) g$, i.e, via two consecutive applications
(compositions) of the pseudodifferential operator
\begin{equation}\label{eq:HelScaPsi}
  \Psi \coloneqq \isrHelmu
\end{equation}
applied to the source function $g$. In practice, we must solve two IBVP of the
form~\eqref{eq:ParaxialEqn} sequentially. For the first IBVP, the source term
$g$ is used as the initial condition for the variable-coefficient paraxial
equation~\eqref{eq:ParaxialEqn}. The OFT integrates this solution in
pseudo-time to compute an approximation to $v_1 =  \Psi g$ at all points of the
domain $\Dom$. The second IBVP that must be solve is identical
to~\eqref{eq:ParaxialEqn} except the initial condition is set to $v_1$ in that
case. The solution to the second IBVP is then used in the cumulative evaluation
of the OFT to produce the final approximation to $\vSca = \Psi v_1 = \Psi^2 g$.

For this example, we consider two wavelengths. In the first case, $\lambda =
0.4$ so that $\kappa = 2\pi / 0.4 \approx 15.7$ and the computational domain
consists of $500 \times 500 = 250,000$ uniformly-spaced points. The spatial
step size along each dimension is $\Delta x = \Delta x_1 = \Delta x_2 = 2\cdot
10^{-2}$. The exponential time-stepping parameters were set to $\Delta t_0 =
10^{-4}$, $\Delta t_T = 5\cdot 10^{-3}$, and $T = \kappa L$ where $L = 10$ is
the length of the domain.  In the second case, $\lambda = 0.1$ and $\kappa =
2\pi / 0.1 \approx 62.8$, the grid is comprised of $2000 \times 2000 =
4,000,000$ uniformly-spaced points, and the spatial step size is $\Delta x =
\Delta x_1 = \Delta x_2 = 5\cdot 10^{-3}$. For both resolutions, the OFT was
integrated in pseudo-time to a variable-coefficient Helmholtz equation residual
tolerance of $10^{-2}$.

Figure~\ref{fig:AFLogo}(a) and Figure~\ref{fig:AFLogo}(b) show the total field
magnitude $|\vTot| = |\vInc + \vSca|$ that result from the $25$-wavelength and
$100$-wavelength incident plane waves traveling from left to right. The waves
are partially reflected and partially transmitted through the inclusion.
Scattering of the smaller wavelength wave through the obstacle results in a
minimum and maximum field magnitude of $0.477$ and $1.388$, respectively. The
$100$-wavelength incident wave scattering leads to both lower and higher field
magnitudes of $0.002$ and $1.995$.  Elevated pseudocolor plots of the total
field magnitude for the $25\lambda$ and $100\lambda$ plane wave scattering
cases are shown in Figure~\ref{fig:AFLogo}(c) and Figure~\ref{fig:AFLogo}(d),
respectively.  These views highlight complex multiple wave scattering through
and around the refractive index perturbation.

\subsection{Three-dimensional plane wave scattering from turbulent 
channel flow}\label{sec:Turbulence}
%
\begin{figure}[h!]
  \centering
  \includegraphics[width=\textwidth]{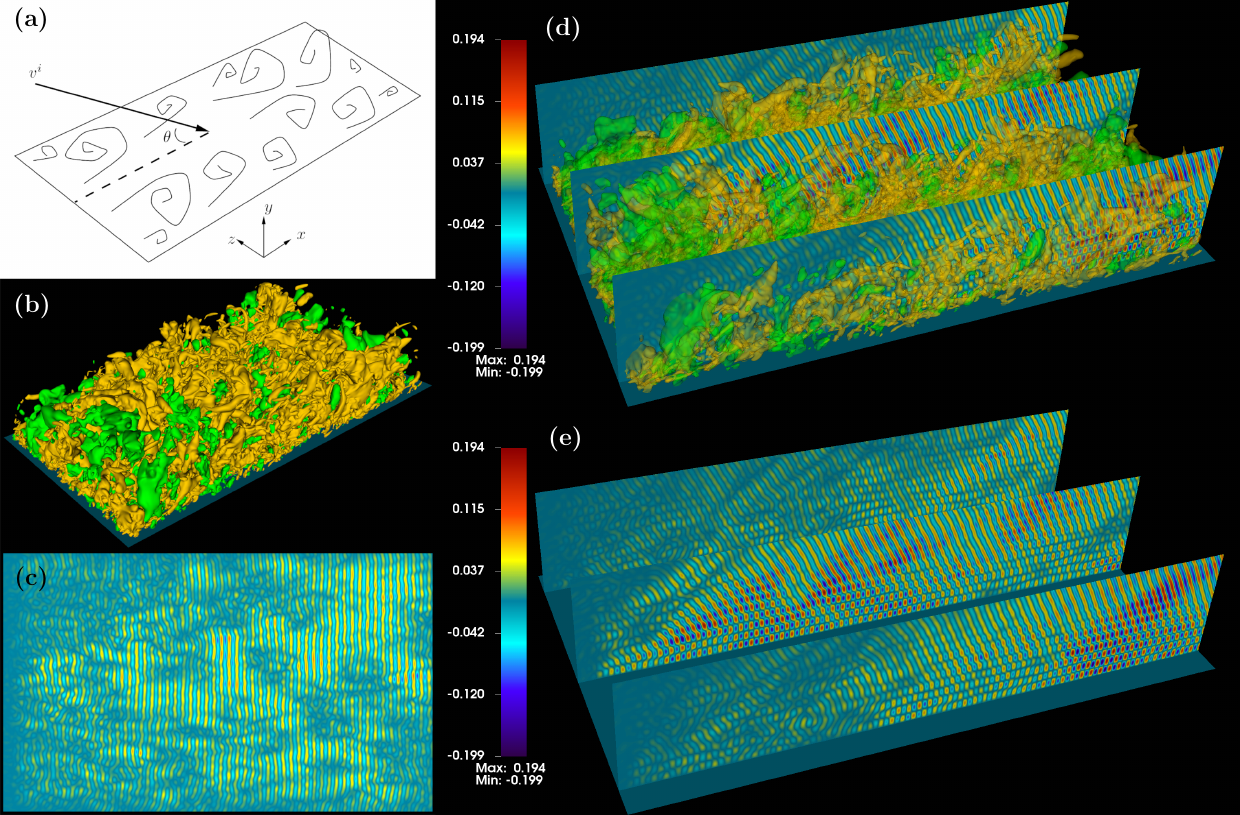}
  \caption{Plane wave scattering from turbulent channel flow. \emph{(a)} 
           Problem schematic of a plane wave with wavenumber $\kappa = 78.5$ 
           traveling in the direction $(1, -0.5, 0) / \sqrt{1.25}$ impinging 
           on the channel wall located at $x_2 = 0$. \emph{(b)} Variable 
           refraction coefficient $\refco(\bx)$ derived from turbulent flow data.
           \emph{(c)} Horizontal cross-section pseudocolor plot at $x_2=0.25$ of
           the real part of the scattered field $\vSca$. \emph{(d)-(e)}
           Vertical pseudocolor planes with and without the turbulent refraction
           inhomogeneity for the real part
           of the scattered field $\vSca$ which show the multiple wave
           scattering that results from interaction with the channel wall and the
           complex surrounding inhomogeneous medium.} 
  \label{fig:ChannelScat}
\end{figure}
Next, we demonstrate our numerical algorithms with a simulation of
three-dimensional plane wave propagation through turbulent flow over a channel
wall. The turbulent flow data was obtained from the Johns Hopkins Turbulence
Databases~\cite{LiEtAl2008,PerlmanEtAl2007,GrahamEtAl2016} and consists of a
direct numerical simulation (DNS) of channel flow in a domain $[0,8\pi) \times
[-1,1] \times [0,3\pi)$, obtained using $2048 \times 512 \times 1536$ nodes.
The friction velocity $u_{\tau} = 0.0499$, and friction velocity Reynolds
number $\textrm{Re}_{\tau} \approx 1000$. Additional turbulence flow parameter
details can be found in the reference~\cite{GrahamEtAl2016}. Although the
turbulence simulation is incompressible, we use the thermodynamic relations
\begin{equation}
  p = p_0 + \wt p = \rho^\gamma, \quad \frac{dp}{d\rho} = c^2, 
\end{equation}
to obtain a spatially-dependent refraction coefficient $\refco(\bx) = \big( c_0 /
c(\bx) \big)^2$ from the turbulent pressure data $\wt p(\bx)$, taking
$p_0 = 1$, $\gamma = 1.4$, and $c_0$ to be the average of $c(\bx)$ over the
$x_2=0$ plane of the channel.

We use only a subset of the turbulence data, which we interpolate and shift to
our domain of interest $\tilde{\Dom} = [-2.5,2.5] \times [0,1] \times
[-1.5,1.5]$.  The channel wall is located at $x_2 = 0$ and the incident plane
wave that impinges on the wall has the form $\tilde{v}^{i}(\bx) = e^{i \bk
\cdot \bx }$, where $\bk = (\kappa_1,\kappa_2,\kappa_3) = |\bk| (1, -0.5, 0) /
\sqrt{1.25}$, $|\bk| = \kappa = 78.5$.  To pose the wall scattering problem in
the same form as~\eqref{eq:HelmScat}, which assumes non-reflecting boundary
conditions on all sides of the computational domain, we use the method of
images so that the incident field used in computations is 
\begin{equation}\label{eq:ChaFlwIncFld}
  \vInc(x_1,x_2,x_3) = e^{i (\kappa_1 x_1 + \kappa_2 x_2 + \kappa_2 x_3)}
                     - e^{i (\kappa_1 x_1 - \kappa_2 x_2 + \kappa_2 x_3)}
\end{equation}
and the domain $\tilde{\Dom}$ is reflected about $x_2 = 0$ so that $\Dom =
[-2.5,2.5] \times [-1,1] \times [-1.5,1.5]$. Note that
solving~\eqref{eq:HelmScat} over $\Dom$ and using $\vInc$ as defined
in~\eqref{eq:ChaFlwIncFld} has the effect of implicitly imposing zero Dirichlet
conditions at $x_2=0$. The number of discretization points used for $\Dom$ is
$1442 \times 722 \times 963$, so that we must solve for $1,002,602,412$ complex
unknowns; the exponential time-stepping parameters were set to $T = \kappa L,\,
\Delta t_0 = 2\cdot 10^{-4}$ and $\Delta t_T = 4 \cdot 10^{-2}$, where the
longest length of the domain $L=5$. The OFT was integrated in pseudo-time for
$131,600$ steps to satisfy a variable-coefficient Helmholtz equation residual
tolerance of $10^{-2}$.

Figure~\ref{fig:ChannelScat}(a) shows a schematic of the scattering problem
setup and Figure~\ref{fig:ChannelScat}(b) displays the turbulent flow variable
refraction coefficient over the channel wall.  The real part of the scattered
field $\vSca$ over a horizontal cross-section pseudocolor plot at $x_2=0.25$
is shown in Figure~\ref{fig:ChannelScat}(c).  The images in
Figures~\ref{fig:ChannelScat}(d) and~\ref{fig:ChannelScat}(e) present vertical
pseudocolor planes with and without the turbulent refraction inhomogeneity for
the real part of the scattered field $\vSca$; the intricate multiple wave
scattering that results from the impinging plane wave with the channel wall and
the complex surrounding inhomogeneous medium is visible throughout the
computational domain.

\subsection{Plane wave transmission through a Luneburg lens}\label{sec:Luneburg}
%
\begin{figure}[h!]
  \centering
  \includegraphics[width=\textwidth]{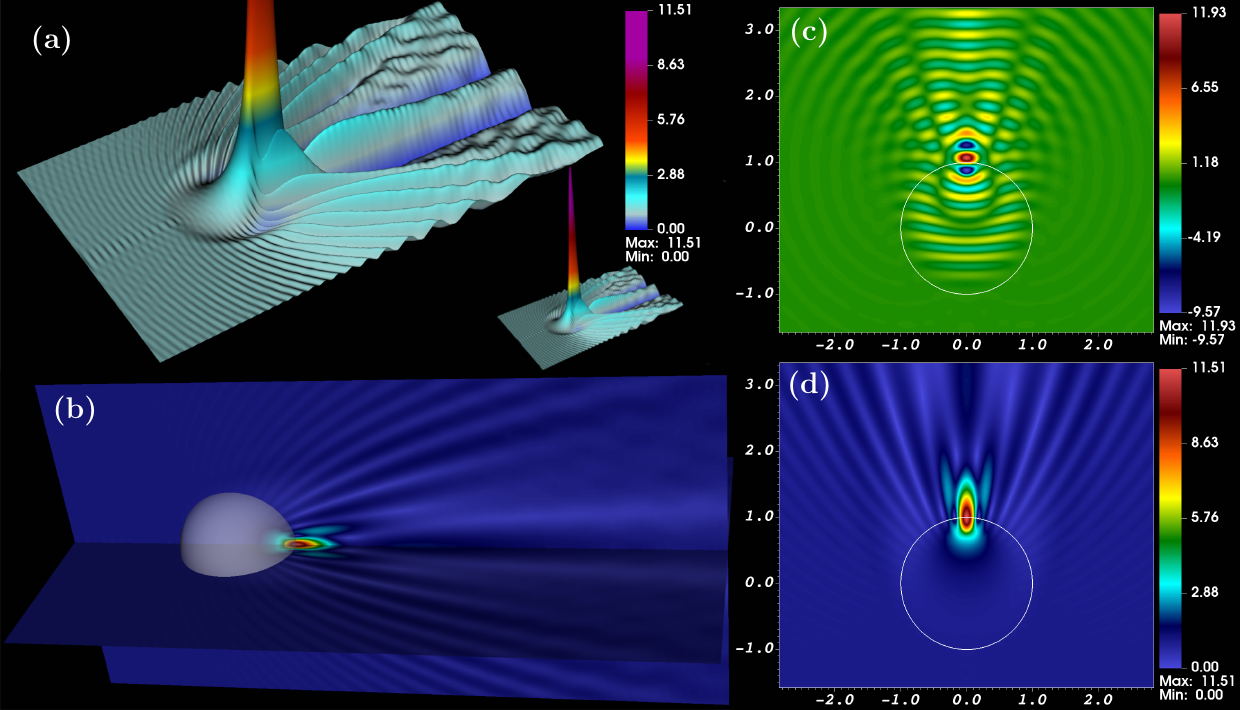}
  \caption{Plane wave transmission through a Luneburg lens. \emph{(a)} A
           plane wave travels from left to right and the total field 
           magnitude $|\vTot|$ peak reveals the focusing of the transmitted
           wave just after it exits the lens. \emph{(b)} Total field magnitude
           cross-sections highlighting the plane wave focusing through the 
           spherical gradient index lens.
           \emph{(c)} Real part of the scattered field 
           with the lens removed reveals the internal wave interaction
           inside the lens. \emph{(d)} 
           Close-up of a cross-section of the total field magnitude $|\vTot|$ 
           which confirms the correct location of the focusing region centered
           on the lens boundary.} 
  \label{fig:Luneburg}
\end{figure}

In this final example we simulate plane wave transmission through a Luneburg
lens. The type of Luneburg lens we consider is a single unit sphere with a
refraction coefficient that varies only along the radial direction.  Luneburg
lenses were first investigated for their electromagnetic radiation focusing
properties~\cite[Ch.3]{Luneburg1966} but, more recently, they have found
applications in acoustic wave manipulation in a diversity of
media~\cite{ZhaoEtAl2023}.

For this example we solve the three-dimensional Helmholtz
equation~\eqref{eq:HelmScat} for the scattered (transmitted) field $\vSca(\bx)$
where the incident field is a plane wave $\vInc(\bx) = e^{i\kappa x_1}$ with
$\kappa = 20.3$.  The domain $\Dom = [-3,3] \times[-3,3] \times [-3,8]$ is
represented with a uniform grid that consists of $818 \times 818 \times 1498 =
1,002,347,752$ points; the spatial step size is approximately equal along every
dimension: $\Delta x = \Delta x_1 = \Delta x_2 \approx \Delta x_3 \approx
7.3\cdot 10^{-3}$. The exponential time-stepping parameters were set to $T =
\kappa L,\, \Delta t_0 = 2\cdot 10^{-4}$ and $\Delta t_T = 4 \cdot 10^{-2}$,
where $L = 11$ (the longest length of the domain). The lens is a unit sphere
centered at the origin and the refraction coefficient in the domain is
\begin{equation}
  \refco(r) = 
  \begin{cases}
  2 - r^2, & 0 \leq r \leq 1, \\
  1,       & r > 1,
  \end{cases}
\end{equation}
where $r = \sqrt{ x_1^2 + x_2^2 + x_3^2 }$. The OFT was integrated in
pseudo-time for $71,400$ steps to satisfy a variable-coefficient Helmholtz
equation residual tolerance of $10^{-2}$.

Figure~\ref{fig:Luneburg}(a) shows plane wave transmission through the 
Luneburg lens; the total field magnitude $|\vTot|$ peak, which achieves a
maximum value of $11.51$, reveals the focusing of the transmitted wave as it
exits the lens.  The two cross-sections of the total field magnitude displayed
in Figure~\ref{fig:Luneburg}(b) highlight the plane wave focusing through the
spherical gradient index lens.  Figures~\ref{fig:Luneburg}(c)
and~\ref{fig:Luneburg}(d) depict cross-sections of the real part of the
scattered field and the total field magnitude, respectively, to reveal the
internal wave interaction inside the lens and confirm the
theoretically-predicted location of the focusing region which should be
centered on the lens boundary.

\section{Conclusions}\label{sec:conclusion}

We introduced novel numerical algorithms for the direct solution of the 1D, 2D,
and 3D Helmholtz equation with a spatially-dependent refraction coefficient
based on an Operator Fourier Transform (OFT) representation of
pseudodifferential operators ($\PDO$).  Our solution approach relies on
expressing the inverse Helmholtz operator in terms of two sequential
applications of an inverse square root pseudodifferential operator. The action
of each inverse square root $\PDO$ on a given function has a simple
representation in the OFT framework: The $\PDO$ is an integral of operators
applied to a function which can be evaluated in terms of solutions of a
pseudo-temporal initial-boundary-value problem for a paraxial equation.  This
framework offers several advantages over traditional iterative approaches for
the Helmholtz equation. The operator integral transform is amenable to standard
quadrature methods and the required pseudo-temporal paraxial equation solutions
can be obtained using any suitable numerical method.

The numerical results presented, which included 2D plane wave scattering from a
complex material inhomogeneity, 3D scattering from turbulent channel flow, and
plane wave transmission through a Luneburg spherical gradient-index
lens, demonstrated that our exponentially-spaced OFT quadrature and our BDF-ADI
finite difference solvers for the paraxial IBVP lead to accurate and efficient
algorithms that can produce solutions of the Helmholtz in non-trivial settings,
even for three-dimensional problems with more than one billion complex
unknowns, on a single present-day workstation.

For simplicity, in this work we presented applications of the OFT using only
2nd order finite differences and a 1st-order implicit time-marching method, but
the OFT methodology is by no means limited to low order numerical algorithms
for the solution of the underlying paraxial equation problems.  In fact,
paraxial equation solvers based on higher-order hybrid implicit-explicit
time-marching schemes and multi-domain spectral spatial approximations
currently in development will demonstrate the applicability of the OFT approach
to problems of scientific and engineering interest that are orders of magnitude
larger than the numerical demonstrations presented in this work.

\subsection*{Acknowledgments}
Approved for public release; distribution is unlimited. AFRL Public Affairs
release approval number AFRL-2024-3739.

The authors gratefully acknowledge support from the Air Force Office of
Scientific Research through grant number 23RDCOR004.

\appendix

\section{Exact solution for non-reflecting boundary conditions} \label{sec:ExactSoln}

In this section we describe the exact solution of the initial-boundary-value
problem 
\begin{equation} \label{eq:IBVP}
  \begin{cases}
  \displaystyle u_t(x,t) 
  = \frac{\ii}{\kappa^2} u_{xx}(x,t), & (x,t) \in (x_{\ell},x_r) \times (0,\infty), \\[1ex]
  \alpha u(x,t) - \ii u_x(x,t) = 0, & (x,t) \in \{ x_{\ell} \} \times (0,\infty), \\
  \alpha u(x,t) + \ii u_x(x,t) = 0, & (x,t) \in \{ x_{r} \}    \times (0,\infty), \\
  u(x,0) = g(x), & x \in [x_{\ell},x_r],
  \end{cases}
\end{equation}
for real $\alpha > 0$ by separation of variables. That is, if we can find a
complete set of eigenvalues $-\lambda_n^2$ and eigenfunctions $\varphi_n(x)$ of
the operator $Lu = u_{xx}$ with the boundary conditions specified above, then
we may expand the initial condition as
\begin{equation} \label{eq:GExpansion}
  g(x) = \sum_{n=1}^{\infty} c_n \varphi_n(x),
\end{equation}
and the solution of equation~\eqref{eq:IBVP} is
\begin{equation} \label{eq:SepOfVarSolution}
  u(x,t) = \sum_{n=1}^{\infty} c_n e^{-\ii (\lambda_n / \kappa)^2 t} \varphi_n(x).
\end{equation}
The pseudodifferential operator applied to $g$ can then be computed as
\begin{equation}
  \frac 1{\sqrt{\kappa^2 + \partial_x^2}} \, g(x) = \sum_{n=1}^{\infty} \frac{c_n}{\sqrt{\kappa^2 - \lambda_n^2}} \varphi_n(x).
\end{equation}
%

\subsection{Eigenvalues and eigenfunctions} \label{sec:EigenProblem}

The eigenvalue problem is
\begin{equation}
  v'' = -\lambda^2 v, \quad
  \begin{cases}
    \ii \alpha v(x_{\ell}) + v'(x_{\ell}) = 0, \\
    \ii \alpha v(x_r) - v'(x_r) = 0,
  \end{cases}
\end{equation}
but the operator is not self-adjoint due to the boundary conditions so the
usual Sturm-Liouville theory does not apply. However, the boundary conditions
are regular in the sense of~\cite{Naimark1968} and thus some of the same
results hold---in particular, the eigenvalues are countable with infinity being
the only accumulation point and the eigenfunctions form a complete basis on the
interval (see~\cite{Naimark1968} for further results). We will see that the
main relevant differences from the usual Sturm-Liouville theory are that the
eigenvalues are complex and the eigenfunctions are not orthogonal.  

To solve the eigenvalue problem, we substitute into the boundary conditions the
ansatz
\begin{equation}
  v(x) = a \cos \lambda (x-x_{\ell}) + b \sin \lambda (x-x_{\ell}),
\end{equation}
which satisfies the differential equation for any constants $a$ and $b$. This
leads to the system of equations
\begin{equation}
\begin{cases}
    \ii \alpha a + b \lambda = 0, \\
    (\ii \alpha a - b \lambda) \cos \lambda L + (\ii \alpha b + a \lambda) \sin \lambda L = 0,
\end{cases}
\end{equation}
where $L = x_r - x_{\ell}$. Let $C = (1+\alpha^2/|\lambda|^2)^{-1/2}$ be a
normalization constant and set $a = C$ and $b = -\ii C \alpha / \lambda$, which
satisfies the first equation. The second equation then gives us the relation
for the eigenvalues,
\begin{equation} \label{eq:LambdaFun}
  f(\lambda) = (\alpha^2 + \lambda^2) \sin L\lambda + 2 \ii \alpha \lambda \cos L\lambda = 0 
\end{equation}
Note that the equations for $\lambda$ and $-\lambda$ are the same, as is the
equation for the eigenfunction in each case, so we may restrict our attention
to eigenvalues with non-negative real part only and index them by positive
integers $n \geq 1$.  Asymptotically as $|\lambda| \rightarrow \infty$ the
equation becomes $\sin \lambda L \sim 0$ which has the solution $\lambda \sim n
\pi / L$.  The eigenfunctions are asymptotically given by $\varphi_n(x) \sim
\cos \big( (n\pi/L) (x - x_{\ell})\big)$. In addition to the eigenvalues
corresponding asymptotically to the zeros of $\sin \lambda L$, there is one
corresponding to the zeros of $(\alpha^2 + \lambda^2)$. The eigenvalue $\lambda
= 0$ has eigenfunction $\varphi = 0$ so it is trivial and we do not include it
in the enumeration. The following propositions show that these are all of the
eigenvalues and that they all have negative imaginary part.

\begin{prop} \label{prop:Rouche}
There are positive numbers $Y_0>\alpha$ and $n_0$ such that for any integer $n >
n_0$ and real $Y > Y_0$ the number of roots of equation~\eqref{eq:LambdaFun} in
the rectangle $R = R(n,Y) = \{x+\ii y \in \C\, :\, |x| < (n+1/2)\pi/L,\, |y| <
Y \}$ is $2n+3$.
\end{prop}

\begin{proof}
This follows from Rouch\'e's theorem. Let $f_1(z) = (\alpha^2 + z^2) \sin Lz$
and $f_2(z) = 2\ii \alpha z \cos Lz$. On the right and left edges of the
rectangle we have $|f_1(z)| = |\alpha^2 + z^2| \cosh LY > |2 \alpha z| \sinh LY
= |f_2(z)|$ for large enough $|z|$. On the top and bottom edges we have
$|f_1(z)| \geq \frac 12 (|z|^2 - \alpha^2) (e^{LY} - 1)$ and $|f_2(z)| \leq
\alpha |z| (e^{LY} + 1)$ for large enough $Y$ and $|z|$, so that $|f_1| >
|f_2|$ for large enough $Y$ and $|z|$.  Therefore, $|f_1| > |f_2|$ on $\partial
R$ so $f = f_1 + f_2$ has the same number of zeros as $f_1$ inside $R$, which
is $2n+3$.
\end{proof}

\begin{figure}[h!]
  \centering
  \includegraphics[width=.5\textwidth]{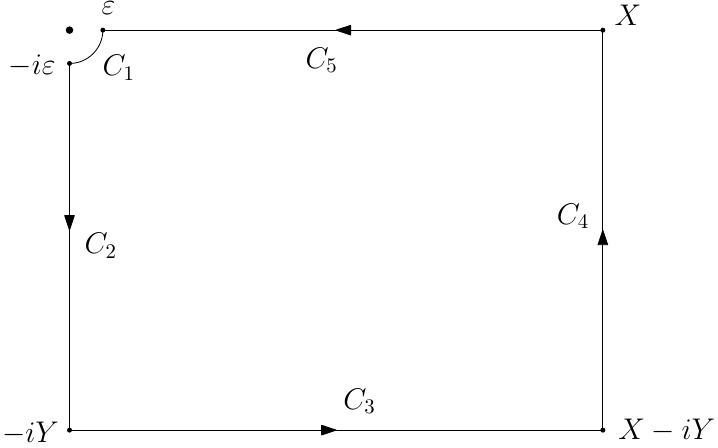}
  \caption{Integration contour for Proposition~\ref{prop:Winding}.} 
  \label{fig:Contour}
\end{figure}

\begin{prop} \label{prop:Winding}
If $z$ is a root of equation~\eqref{eq:LambdaFun} with positive real part, then
the imaginary part of $z$ is negative.
\end{prop}

\begin{proof}
We compute the winding number of the contour $C = C_1 + C_2 + C_3 + C_4 + C_5$
depicted in Fig.~\ref{fig:Contour} and show that it accounts for all the
eigenvalues in the right-half plane as found in Proposition~\ref{prop:Rouche}.
Let $X = (2n + 1/2)\pi/L > \alpha$ for some positive integer $n$ and let $Y >
0$. The contours are given by $C_1 = \{\veps e^{\ii \theta} \,:\,\theta \in
(-\pi/2,0)\}$, $C_2 = \{ -\ii y \,:\, y \in (\veps,Y)\}$, $C_3 = \{ x - \ii Y
\,:\, x \in (0,X) \}$, $C_4 = \{ X - \ii y \,:\, y \in (0,Y)\}$, $C_4 = \{x
\,:\, x \in (\veps,X)\}$ and we consider the limit $\veps \rightarrow 0$, $Y
\rightarrow \infty$.  Let $w_j$ be the winding number along each contour.
Along $C_1$, $f(z) \sim (\alpha^2 L + 2\ii \alpha) z$ as $\veps \rightarrow 0$
so the winding number is simply $w_1 = -1/4$. Along $C_2$ we have 
\begin{align*}
  f(-\ii y) &= -\ii (\alpha^2 - y^2 ) \sinh Ly + 2\alpha y \cosh Ly, \\
  f(-\ii \veps) &\sim 2\alpha \veps - \ii \alpha^2 L\veps, \\
  f(-\ii Y) &\sim \frac{\ii}{2} Y^2 e^{LY}.
\end{align*}
The real part is always positive and the imaginary part is monotonically
increasing, so the winding number is $w_2 = 1/4 + \theta_0/(2\pi)$, where $\theta_0
= |\Arg(2\alpha - \ii \alpha^2L)|$. Along $C_3$ we have 
\begin{align*}
  f(x - \ii Y) \sim \frac 12 Y^2 e^{LY}( -\sin Lx + \ii \cos Lx ).
\end{align*}
From $x = 0$ to $x = X = (2n + 1/2) \pi/L$ we have that the winding number is
$w_3 = n + 1/4$. Along $C_4$ we have
\begin{align*}
  f(X-\ii y) &= (\alpha^2 + X^2 - y^2) \cosh Ly - 2\alpha X \sinh Ly -2\ii y ( X \cosh Ly - \alpha \sinh Ly ), \\
  f(X-\ii Y) &\sim -\frac 12 Y^2 e^{LY}, \\
  f(X) &= \alpha^2 + X^2.
\end{align*}
The imaginary part is always negative for large enough $X$, so the winding
number is $w_4 = 1/2$.  Finally, along $C_5$ we have
\begin{align*}
  f(x) &= (\alpha^2 + x^2) \sin Lx + 2\ii \alpha x \cos Lx, \\
  f(\pi/(2L)) &= \alpha^2 + \Big( \frac{\pi}{2L} \Big)^2, \\
  f(\veps) &\sim \alpha^2 L \veps + 2\ii \alpha \veps.
\end{align*}
From $x = X$ to $x = \pi/(2L)$ there are $n$ oscillations; from $x = \pi/(2L)$
to $x = \veps$ the real and imaginary parts are positive. Since $\alpha^2 L +
2\ii \alpha = \ii( 2\alpha - \ii \alpha^2 L)$, we have $w_5 = n + 1/4 -
\theta_0/(2\pi)$. Therefore, the total winding number for the whole contour is
$w_C = \sum_{j=1}^5 w_j = 2n + 1$, so these are the number of zeros in the
lower-left quadrant of the complex plane. By Proposition~\ref{prop:Rouche} and
the symmetry of $f(z)$, these are all the zeros in the region $\{ z \,:\, 0 <
\re(z) < X \}$.
\end{proof}

The eigenvalues can be found numerically. Once these are obtained, the inner
product of $\varphi_m$ and $\varphi_n$ is given by
\begin{align}
  (\varphi_m,\varphi_n) &= \int_{x_{\ell}}^{x_r} \overline \varphi_m(x) \varphi_n(x) \,dx \nonumber \\
    &= C_m C_n \bigg( \frac{2\ii \alpha (1 - \cos L\lbar_m \cos L\lambda_n)}{\lbar_m^2 - \lambda_n^2} + \frac{(\alpha^2 + \lbar_m^2) \sin L\lbar_m \cos L\lambda_n}{\lbar_m(\lbar_m^2 - \lambda_n^2)} \nonumber \\
    &\quad - \frac{(\alpha^2 + \lambda_n^2) \cos L\lbar_m \sin L\lambda_n}{\lambda_n(\lbar_m^2 - \lambda_n^2)} - \frac{\ii \alpha (\lbar_m^2 + \lambda_n^2) \sin L\lbar_m \sin L\lambda_n}{\lbar_m \lambda_n(\lbar_m^2 - \lambda_n^2)} \bigg).
\end{align}
If we truncate the expansion in equation~\eqref{eq:GExpansion} up to $N$ terms,
then the coefficients are the solution of the $N\times N$ linear system $Ac=b$,
where $c$ is the vector of coefficients and the matrix $A$ and right-hand side
vector $b$ have entries 
\begin{equation}
  A_{mn} = (\varphi_m,\varphi_n), \quad b_m = (\varphi_m,g)
\end{equation}
The inner products for $b$ can be computed numerically using, e.g., F\'ejer
quadrature.

We will see that the long-time behavior of the
solution~\eqref{eq:SepOfVarSolution} will depend primarily on the ``low
frequency'' modes of the expansion. Therefore, we determine the asymptotic
behavior of the eigenvalues when $\alpha \gg \lambda$.  In this regime
equation~\eqref{eq:LambdaFun} is asymptotically $\sin \lambda L \sim 0$, which
has the solution $\lambda \sim n\pi/L$. Let $\lambda_n = (n\pi + \veps_n)/L$
and let $\delta = 1/(\alpha L)$. Substituting into
equation~\eqref{eq:LambdaFun} we have
\begin{equation}
  \tan \veps_n = -\frac{2 \ii \delta (n\pi + \veps_n)}{1 + \delta (n\pi + \veps_n)^2}.
\end{equation}
The right-hand side tends to zero as $\delta \rightarrow 0$, so we may expand
the left-hand side in powers of $\veps_n$. Retaining only first order terms in
$\delta$ and $\veps_n$ we have
\begin{equation} \label{eq:LowFreqEig}
  \veps_n \sim - \ii \frac{2n\pi}{\alpha L} \quad \Longrightarrow \quad -\ii \lambda_n^2 \sim -\Big(\frac{n\pi}L \Big)^2 \Big(  \frac{4}{\alpha L} + \ii \Big).
\end{equation}
Note that $\veps_n \ll 1$ implies $n \ll \alpha L$; this defines the low
frequency regime. To first order in $\delta$, the eigenfunctions in this regime
are given by
\begin{align}
  \varphi_n(x) \sim \frac{-\ii \alpha/\lambda_n}{\sqrt{1+\alpha^2/|\lambda_n|^2}} \sin \lambda_n(x-x_{\ell}) \sim \Big( \frac 2{\alpha L} - \ii \Big) \sin \lambda_n(x-x_{\ell})
\end{align}
%

\subsection{Paraxial equation asymptotic solution} \label{sec:ParaxialAsymptotics}

We now set $\alpha = \kappa$ (which is the value of interest for non-reflecting
boundary conditions of the Helmholtz equation) and derive the asymptotic
behavior of the solution~\eqref{eq:SepOfVarSolution} in the regime $t > (\kappa
L)^2/(2\pi)$, $\kappa L \gg 1$. We require only the lowest order algebraic
term, as well as the lowest order terms in both the real and imaginary parts of
exponentials.  For large $t$, the exponential term inside the sum decays and
oscillates rapidly with increasing $n$, so we expect the sum in the solution to
depend mostly on a neighborhood of $n=1$. Let $\delta = \kappa L$ and $X = (x -
x_{\ell})/L$. The eigenfunctions are simply $\varphi_n(x) \sim -\ii \sin\big(
\lambda_n X/(\kappa\delta) \big)$.  Writing sine as a sum of exponentials,
combining these with the time dependent exponential part of the solution and
completing the square we have
\begin{equation}
  -\ii t \Big(\frac{\lambda_n}{\kappa} \Big)^2 \pm \ii \frac{\lambda_n}{\kappa} \frac{X}{\delta} = \ii \frac{X^2}{4\delta^2 t} - t \theta_n^{\pm}, \quad \theta_n^{\pm} = \ii \Big( \frac{\lambda_n}{\kappa} \mp \frac{X}{2\delta t} \Big)^2.
\end{equation}
Therefore, we may write the solution as
\begin{equation} \label{eq:UAsymSSum}
  u(x,t) = \frac 12 e^{\ii \frac{X^2}{4\delta^2 t}} ( S^- - S^+ ), \quad S^{\pm} = \sum_{n=1}^{\infty} c_n e^{-t \theta_n^{\pm}}. 
\end{equation}
In the low-frequency regime, the eigenvalues take the form
\begin{equation}
  \frac{\lambda_n}{\kappa} \sim n\pi \delta (1-2\ii \delta).
\end{equation}
Let $s = s(n) = n\pi\delta$, $\Delta s = \pi\delta$, and $\phi^{\pm}(s) = \theta^{\pm}_n$,
which is asymptotically
\begin{equation}
  \phi^{\pm}(s) \sim (4\delta + \ii) \Big( s \mp (1+2\ii\delta) \frac{X}{2\delta t} \Big)^2.
\end{equation}
Then the summations are given by
\begin{equation}
  S^{\pm} \sim \frac 1{\pi \delta} \sum_{n=1}^{\infty} c_n e^{-t\phi^{\pm}(s)} \Delta s.
\end{equation}
Assume that the coefficients vary slowly such that $c_{n+1} - c_n = \mathcal
O(\delta)$ and for simplicity also assume that $c_1 \neq 0$. Let $\gamma(s)$ be
a smooth continuous approximation of the coefficients, the details of which are
unimportant when $c_1 \neq 0$ (see Remark~\ref{rem:GeneralCoeffs} for the
general case). Then the summations are an approximation of the integrals
\begin{equation} \label{eq:SIntegral}
  S^{\pm} \sim \frac 1{\pi \delta} \int_{\pi\delta}^{\infty} \gamma(s) e^{-t\phi^{\pm}(s)}  \,ds.
\end{equation}
Asymptotically, the main contribution to the integral is around the endpoint $s
= \pi\delta$, under the assumption $t > (2\pi \delta^2)^{-1}$. Therefore, the
integral is asymptotically given by

\begin{align}
  S^{\pm} &\sim \frac{\gamma(\pi \delta)}{\pi \delta} \int_{\pi\delta}^{\infty} e^{-t\phi^{\pm}(s)} \,ds \nonumber \\
    &= \frac{c_1}{\pi \delta} \frac 12 \sqrt{ \frac{\pi}{(4\delta + \ii)t} } \erfc \Big( \sqrt{ (4\delta + \ii)t^{-1} }\big(\pi\delta t \mp ( (2\delta)^{-1} + \ii) X \big) \Big) \nonumber \\
    &\sim \frac{c_1 e^{-\ii \frac{X^2}{4\delta^2 t}}}{2\pi\delta^2 t \mp X} e^{\pm\ii\pi X} e^{-(\pi\delta)^2(4\delta + \ii) t}.
\end{align}
Substituting this into equation~\eqref{eq:UAsymSSum}, we have
\begin{equation} \label{eq:UAsym}
  u(x,t) \sim \frac{c_1}{2\pi\delta^2 t} e^{-(\pi\delta)^2(4\delta + \ii) t} \Big( \frac{e^{-\ii\pi X}}{2+X/(\pi\delta^2 t)} - \frac{e^{\ii\pi X}}{2-X/(\pi\delta^2 t)} \Big).
\end{equation}
It follows that in $d$ dimensional space the paraxial solution in a rectangle with dimensions $L_1 \times \dots \times L_d$ is asymptotically
\begin{equation} \label{eq:UAsymDDim}
  u(\bx,t) \sim t^{-d} e^{-(a_d + \ii b_d) t} h(\bx,t),
\end{equation}
where $h(\bx,t)$ is slowly-varying in $t$ and
\begin{equation} \label{eq:AdBd}
a_d = \frac{4\pi^2}{\kappa^3} \sum_{\ell=1}^d \frac 1{L_\ell^3}, \quad b_d = \frac{\pi^2}{\kappa^2} \sum_{\ell=1}^d \frac 1{L_\ell^2}
\end{equation}

\begin{rem} \label{rem:GeneralCoeffs}
For general expansion coefficients $c_n$, we define $\gamma(s)$ using a Taylor
series approximation centered at $s = \Delta s$ ($= \pi \delta$) with the
derivatives of $\gamma(s)$ given by forward divided differences of the
coefficients:
\begin{align}
  \gamma(s) &= \sum_{k=0}^{\infty} \frac{\Delta^k c_1}{k!} (s-\Delta s)^k, \\
  \Delta^0 c_1 &= c_1, \quad \Delta^{k+1} c_1 = \frac{\Delta^k c_2 - \Delta^k c_1}{\Delta s}, \quad k \geq 1.
\end{align}
The integral in equation~\eqref{eq:SIntegral} is then asymptotically given by
applying a form of Watson's lemma for complex phase function.
\end{rem}

\section{Proof of error estimate for numerical OFT} \label{sec:ErrorProof}

In this appendix we present the proof of the error estimate given in
Section~\ref{sec:Error}.

\begin{proof}[Proof of Theorem~\ref{thm:Error}]
In Section~\ref{sec:OFTQuad} we defined the piecewise linear interpolation
operator for a sequence of functions. We extend that definition to a sequence
of operators $\{P_m\}$ by defining for any function $v$
\begin{equation}
  \I_n(P_m) v = \I_n( P_m v )
\end{equation}
We also extend $\I_n$ to continuous-time operators $P(t)$ by defining
$\I_n\big(P(t)\big) = \I_n\big( P(t_m) \big)$. We write the error as
\begin{align}
  \| f(A) - f^N(A) \| &= \Big\| \sqrt{\frac{-\ii}{\pi}} \intzeroinf \frac{e^{\ii\tau}}{\sqrt{\tau}} e^{\ii\tau (A-I)} \,d\tau - \sum_{n=0}^N \omega_n \wt S_n \Big\| \nonumber \\
    &= \frac 1{\sqrt{\pi}} \Big\| \sum_{n=0}^{N-1} \int_{t_n}^{t_{n+1}} \frac{e^{\ii\tau}}{\sqrt{\tau}} \big( e^{\ii\tau (A-I)} - \I_n(\wt S_m) \big) \,d\tau  + \int_{T}^{\infty} \frac{e^{\ii\tau}}{\sqrt{\tau}} e^{\ii\tau (A-I)} \,d\tau \Big\| \nonumber \\
    &\leq \frac{1}{\sqrt{\pi}} \sum_{n=0}^{N-1} \int_{t_n}^{t_{n+1}} \frac{1}{\sqrt{\tau}} \Big\| e^{\ii\tau (A-I)} - \I_n\big( e^{\ii\tau(A-I)} \big) \Big\| \,d\tau \nonumber \\
    &\quad + \frac{1}{\sqrt{\pi}} \sum_{n=0}^{N-1} \int_{t_n}^{t_{n+1}} \frac{1}{\sqrt{\tau}} \Big\| \I_n\big( e^{\ii\tau (A-I)} - \wt S_m \big) \Big\| \,d\tau \nonumber \\
    &\quad + \frac{1}{\sqrt{\pi}} \int_{T}^{\infty} \frac{1}{\sqrt{\tau}} \big\| e^{\ii\tau (A-I)} \big\| \,d\tau.
\end{align}
Let $E_1$, $E_2$, and $E_3$ denote the three terms above, in order. We bound
each of these terms. The third one is bounded simply by
\begin{equation}
  E_3 = \frac{1}{\sqrt{\pi}} \int_{T}^{\infty} \frac{1}{\sqrt{\tau}} \big\| e^{\ii\tau (A-I)} \big\| \,d\tau \leq \frac{\erfc\big( \sqrt{ \sigma T } \big)}{\sqrt{\sigma}} \sim \frac{e^{-\sigma T}}{\sigma \sqrt{\pi T}}.
\end{equation}

For the first term, we use the interpolation error formula over the interval
$(t_n,t_{n+1})$,
\begin{align}
\big\| \big( e^{\ii t (A-I)} - \I_n\big( e^{\ii t(A-I)} \big) \big) v \big\| &\leq \frac 12 (t-t_n)(t_{n+1}-t) \max_{(t_n,t_{n+1})} \Big\|\frac{d^2}{dt^2} e^{\ii t (A-I)} v \Big\| \nonumber \\
  &\leq \frac 12 (t-t_n)(t_{n+1}-t) e^{-\sigma t_n} \| A-I \|^2 \|v\|.
\end{align}
Therefore,
\begin{align}
  E_1 &\leq \sum_{n=0}^{N-1} \frac{e^{-\sigma t_n} \| A-I \|^2}{2\sqrt{\pi}} \int_{t_n}^{t_{n+1}} \frac{(\tau-t_n)(t_{n+1}-\tau)}{\sqrt{\tau}}\, d\tau \nonumber \\
\end{align}
For the first interval we have
\begin{equation}
  \int_{0}^{\Delta t_0} \frac{(\tau-t_n)(t_{n+1}-\tau)}{\sqrt{\tau}} \, d\tau = \frac 4{15}\Delta t_0^{5/2},
\end{equation}
and the remaining intervals are bounded by
\begin{equation}
  \int_{t_n}^{t_{n+1}} \frac{(\tau-t_n)(t_{n+1}-\tau)}{\sqrt{\tau}} \, d\tau \leq \frac 1{\sqrt{t_n}} \frac{\Delta t_n^3}{6},
\end{equation}
Thus we have
\begin{equation} \label{eq:d1Bound}
  E_1 \leq \frac{\| A-I \|^2}{2 \sqrt{\pi}} \bigg( \frac{4}{15} \Delta t_0^{5/2} + \frac{1}{6} \sum_{n=1}^{N} \frac{e^{-\sigma t_n} \Delta t_n^3}{\sqrt{t_n}} \bigg),
\end{equation}
where we added the $N$th term to the summation to simplify what follows.
Writing the last term as a lower Riemann sum, we bound it by 
\begin{align}
  \sum_{n=1}^{N} \frac{e^{-\sigma t_n} \Delta t_n^3}{\sqrt{t_n}} &= \frac{b^3\Delta t_0^3}{\sqrt{a}} \sum_{n=1}^{N} \frac{e^{-\sigma a(b^n-1)} b^{3(n-1)}}{\sqrt{b^n-1}} \nonumber \\
    &\leq \frac{b^3\Delta t_0^3}{\sqrt{a}} \int_0^{N} \frac{e^{-\sigma a(b^n-1)} b^{3n}}{\sqrt{b^n-1}} \,dn \nonumber \\
    &\leq \frac{b^3\Delta t_0^3}{\sqrt{a}} \int_0^{\infty} \frac{e^{-\sigma a(b^n-1)} b^{3n}}{\sqrt{b^n-1}} \,dn \nonumber \\
    &= \bigg( \frac{b\Delta t_0}{\sigma a} \bigg)^3 \frac{\sqrt{\pi \sigma}}{4 \log b} \big( 3 + 4 \sigma a (1+\sigma a) \big).
\end{align}
The asymptotic behavior is straightforward from this formula.

For the second term, we use the Taylor formula
\begin{align}
  e^{\ii t_n (A-I)} v &= e^{\ii t_{n+1} (A-I)} v - \ii \Delta t_n (A-I) e^{\ii t_{n+1} (A-I)} v - \frac{\Delta t_n^2}{2} (A-I)^2  e^{\ii \xi_n (A-I)} v \nonumber \\
    &= \big( I - \ii \Delta t_n (A-I) \big) e^{\ii t_{n+1} (A-I)} v - \frac{\Delta t_n^2}{2} (A-I)^2  e^{\ii \xi_n (A-I)} v, \label{eq:ExpTaylor}
\end{align}
where $\xi_n \in [t_n,t_{n+1}]$. We write the recurrence for the approximate
solution operator $\wt S_n$ as
\begin{equation}
  \wt S_n = \big(I - \ii \Delta t_n (A-I) \big) \wt S_{n+1},
\end{equation}
subtract it from equation~\eqref{eq:ExpTaylor}, and multiply by $B_n$ to get
the recurrence for the error,
\begin{equation}
\big( S_{n+1} - \wt S_{n+1} \big) v = B_n \big( S_n - \wt S_n \big) v + \frac{\Delta t_n^2}{2} (A-I)^2  B_n S(\xi_n) v.
\end{equation}
where $S_n = S(t_n)$ is the solution operator at time $t = t_n$.
Using this recurrence, we have
\begin{equation}
  \big( S_n - \wt S_n \big) v = \frac{(A-I)^2}{2} \sum_{k=0}^{n-1} \Delta t_k^2 \prod_{j=k}^{n-1} B_j S(\xi_k) v.
\end{equation}
Let $\rho_0 = \rho(\Delta t_T)$. It follows that the error is bounded by
\begin{align}
  \big\| S_n - \wt S_n \big\| &\leq E_2^{(n)} = \frac{\|A-I\|^2}{2} \sum_{k=0}^{n-1} \Delta t_k^2 e^{-\sigma t_k} \prod_{j=k}^{n-1} \frac{1}{1+\rho_0 \Delta t_j}.
\end{align}
Asymptotically, we have
\begin{equation}
  \prod_{j=k}^{n-1} \frac{1}{1+\rho_0 \Delta t_j} = \prod_{j=k}^{n-1} \big( e^{-\sigma \Delta t_j} + \Ord( \Delta t_j^2 ) \big) = e^{-\sigma(t_n - t_k)} + \Ord\bigg( \sum_{j=k}^{n-1} \Delta t_j^2 \bigg).
\end{equation}
Therefore,
\begin{align}
  E_2^{(n)} &\sim \frac{\|A-I\|^2}{2} e^{-\sigma t_n} \sum_{k=0}^{n-1} \Delta t_k^2.
\end{align}
The summation can be bounded by
\begin{align}
  \sum_{k=0}^{n-1} \Delta t_k^2 &= \Delta t_0^2 \sum_{k=1}^{n} b^{2(k-1)} \leq \Delta t_0^2 \int_0^{n} b^{2k} \,dk = \frac{t_{2n} \Delta t_0^2}{2a \log b}
\end{align}
We use this estimate in the term,
\begin{align}
  \Big\| \int_{t_n}^{t_{n+1}} \frac{1}{\sqrt{\tau}} \I_n \big( S(\tau) - \wt S_n \big)\, d\tau \Big\| &\leq E_2^{(n)} \int_{t_n}^{t_{n+1}} \frac{t_{n+1}-\tau}{\Delta t \sqrt{\tau}}\, d\tau + E_2^{(n+1)} \int_{t_n}^{t_{n+1}} \frac{\tau-t_n}{\Delta t \sqrt{\tau}}\, d\tau.
\end{align}
For the first interval we have
\begin{equation}
  \int_{0}^{t_1} \frac{t_1-\tau}{\Delta t_0 \sqrt{\tau}}\, d\tau = \frac{4\sqrt{\Delta t_0}}{3}, \quad \int_{0}^{t_1} \frac{\tau}{\Delta t_0 \sqrt{\tau}}\, d\tau = \frac{2\sqrt{\Delta t_0}}{3},
\end{equation}
and for the remaining intervals we have the bound
\begin{equation}
  \int_{t_n}^{t_{n+1}} \frac{t_{n+1}-\tau}{\Delta t_n \sqrt{\tau}}\, d\tau \leq \frac{\Delta t_n}{2\sqrt{t_n}}, \quad \int_{t_n}^{t_{n+1}} \frac{\tau-t_n}{\Delta t_n \sqrt{\tau}}\, d\tau \leq \frac{\Delta t_n}{2\sqrt{t_n}}.
\end{equation}
Therefore,
\begin{align}
  E_2 &= \frac{1}{\sqrt{\pi}} \sum_{n=0}^{N-1} \Big\| \int_{t_n}^{t_{n+1}} \frac{1}{\sqrt{\tau}} \I_n \big( e^{\ii \tau(A-I)} - \wt S_{n} \big)\, d\tau \Big\| \nonumber \\
    &\sim \frac{\|A-I\|^2}{\sqrt{\pi}} \bigg( \frac{\Delta t_0^{5/2}}{3} + \frac{\Delta t_0^2}{2a \log b} \sum_{n=1}^{N-1} \frac{e^{-\sigma t_n} t_{2n} \Delta t_n}{\sqrt{t_n}} \bigg) \label{eq:E2Bound}
\end{align}
The sum has the asymptotic bound
\begin{align}
  \frac{1}{a} \sum_{n=1}^{N-1} \frac{e^{-\sigma t_n} t_{2n} \Delta t_n}{\sqrt{t_n}} &= \frac{\Delta t_0}{\sqrt{a}} \sum_{n=1}^{N-1} \frac{e^{-\sigma a (b^n-1)}(b^{2n}-1) b^n}{\sqrt{b^n - 1}} \nonumber \\
    &\leq \frac{b^3 \Delta t_0}{\sqrt{a}} \sum_{n=1}^{N} \frac{e^{-\sigma a (b^n-1)}(b^{2(n-1)}-b^{-1}) b^{n-1}}{\sqrt{b^n-1}} \nonumber \\
    &\leq \frac{b^3 \Delta t_0}{\sqrt{a}} \int_0^N \frac{e^{-\sigma a (b^n-1)}(b^{2n}-b^{-1}) b^{n}}{\sqrt{b^n-1}} \,dn \nonumber \\
    &\leq \frac{b^3 \Delta t_0}{\sqrt{a}} \int_0^{\infty} \frac{e^{-\sigma a (b^n-1)}(b^{2n}-b^{-1}) b^{n}}{\sqrt{b^n-1}} \,dn \nonumber \\
    &= \frac{\sqrt{\pi}b^2 \Delta t_0}{4a^3\sigma^{5/2}\log b} \big( 3b + 4b a \sigma + 4a^2 (b-1) \sigma^2 \big) \nonumber \\
    &\sim \frac{\sqrt{\pi} R}{4T\sigma^{3/2}}.
\end{align}
Substituting this into equation~\eqref{eq:E2Bound} with $\log b \sim
R\Delta t_0/T$ completes the proof.
\end{proof}

\bibliographystyle{amsplain}
\bibliography{main} 

\providecommand{\bysame}{\leavevmode\hbox to3em{\hrulefill}\thinspace}
\providecommand{\MR}{\relax\ifhmode\unskip\space\fi MR }
\providecommand{\MRhref}[2]{%
  \href{http://www.ams.org/mathscinet-getitem?mr=#1}{#2}
}
\providecommand{\href}[2]{#2}
\begin{thebibliography}{10}

\bibitem{Abels2011}
{H}elmut Abels, \emph{Pseudodifferential and singular integral operators: an
  introduction with applications}, Walter de Gruyter, 2011.

\bibitem{AppeloEtAl2020}
Daniel Appelo, Fortino Garcia, and Olof Runborg, \emph{Wave{H}oltz: {I}terative
  solution of the {H}elmholtz equation via the wave equation}, SIAM Journal on
  Scientific Computing \textbf{42} (2020), no.~4, A1950--A1983.

\bibitem{BanjaiHackbusch2008}
Lehel Banjai and Wolfgang Hackbusch, \emph{Hierarchical matrix techniques for
  low-and high-frequency {H}elmholtz problems}, IMA Journal of Numerical
  Analysis \textbf{28} (2008), no.~1, 46--79.

\bibitem{BaylissEtAl1985}
Alvin Bayliss, Charles~I Goldstein, and Eli Turkel, \emph{The numerical
  solution of the {H}elmholtz equation for wave propagation problems in
  underwater acoustics}, Computers \& Mathematics with Applications \textbf{11}
  (1985), no.~7-8, 655--665.

\bibitem{BrunoEtAl2019}
Oscar~P Bruno, Max Cubillos, and Edwin Jimenez, \emph{Higher-order
  implicit-explicit multi-domain compressible {N}avier-{S}tokes solvers},
  Journal of Computational Physics \textbf{391} (2019), 322--346.

\bibitem{VisIt2012}
Hank Childs, Eric Brugger, Brad Whitlock, Jeremy Meredith, Sean Ahern, David
  Pugmire, Kathleen Biagas, Mark~C Miller, Cyrus Harrison, Gunther~H Weber,
  Hari Krishnan, Thomas Fogal, Allen Sanderson, Christoph Garth, E~Wes Bethel,
  David Camp, Oliver Rubel, Marc Durant, Jean~M Favre, and Paul Navratil,
  \emph{{High Performance Visualization--Enabling Extreme-Scale Scientific
  Insight}}, 2012.

\bibitem{CubillosJimenez2024}
Max Cubillos and Edwin Jimenez, \emph{A numerical functional calculus with
  applications to optical wave propagation}, Book of {A}bstracts, {T}he 16th
  {I}nternational {C}onference on {M}athematical and {N}umerical {A}spects of
  {W}ave {P}ropagation ({WAVES} 2024) (Laurent Gizon, ed.), Edmond MPDL, 2024,
  pp.~367--368.

\bibitem{EngquistYing2011a}
Bj{\"o}rn Engquist and Lexing Ying, \emph{Sweeping preconditioner for the
  {H}elmholtz equation: hierarchical matrix representation}, Communications on
  Pure and Applied Mathematics \textbf{64} (2011), no.~5, 697--735.

\bibitem{EngquistYing2011b}
\bysame, \emph{Sweeping preconditioner for the {H}elmholtz equation: moving
  perfectly matched layers}, Multiscale Modeling \& Simulation \textbf{9}
  (2011), no.~2, 686--710.

\bibitem{EngquistZhao2018}
Bj{\"o}rn Engquist and Hongkai Zhao, \emph{Approximate separability of the
  {G}reen's function of the {H}elmholtz equation in the high frequency limit},
  Communications on Pure and Applied Mathematics \textbf{71} (2018), no.~11,
  2220--2274.

\bibitem{Erlangga2008}
Yogi~A Erlangga, \emph{Advances in iterative methods and preconditioners for
  the {H}elmholtz equation}, Archives of Computational Methods in Engineering
  \textbf{15} (2008), 37--66.

\bibitem{ErnstGander2011}
Oliver~G Ernst and Martin~J Gander, \emph{Why it is difficult to solve
  {H}elmholtz problems with classical iterative methods}, Numerical Analysis of
  Multiscale Problems (2011), 325--363.

\bibitem{GanderZhang2019}
Martin~J Gander and Hui Zhang, \emph{A class of iterative solvers for the
  {H}elmholtz equation: Factorizations, sweeping preconditioners, source
  transfer, single layer potentials, polarized traces, and optimized {S}chwarz
  methods}, SIAM Review \textbf{61} (2019), no.~1, 3--76.

\bibitem{GillmanEtAl2015}
Adrianna Gillman, Alex~H Barnett, and Per-Gunnar Martinsson, \emph{A spectrally
  accurate direct solution technique for frequency-domain scattering problems
  with variable media}, BIT Numerical Mathematics \textbf{55} (2015), 141--170.

\bibitem{GrahamEtAl2016}
J~Graham, K~Kanov, XIA Yang, M~Lee, N~Malaya, CC~Lalescu, R~Burns, G~Eyink,
  A~Szalay, RD~Moser, et~al., \emph{A web services accessible database of
  turbulent channel flow and its use for testing a new integral wall model for
  {LES}}, Journal of Turbulence \textbf{17} (2016), no.~2, 181--215.

\bibitem{IhlenburgBabuska1995Disp}
Frank Ihlenburg and Ivo Babu{\v{s}}ka, \emph{Dispersion analysis and error
  estimation of {G}alerkin finite element methods for the {H}elmholtz
  equation}, International journal for numerical methods in engineering
  \textbf{38} (1995), no.~22, 3745--3774.

\bibitem{IhlenburgBabuska1995FEM}
\bysame, \emph{Finite element solution of the {H}elmholtz equation with high
  wave number {P}art {I}: The h-version of the {FEM}}, Computers \& Mathematics
  with Applications \textbf{30} (1995), no.~9, 9--37.

\bibitem{KeefeEtAl2024}
Laurence Keefe, Austin McDaniel, Max Cubillos, Timothy Madden, and Ilya
  Zilberter, \emph{A vector {H}elmholtz electromagnetic wave propagator for
  inhomogeneous media}, Submitted for publication, 2024.

\bibitem{KeefeEtAl2018}
Laurence Keefe, Ilya Zilberter, and Timothy~J Madden, \emph{When parabolized
  propagation fails: a matrix square root propagator for {EM} waves}, AIAA
  2018-3113, 2018 Plasmadynamics and Lasers Conference, 2018.

\bibitem{LiEtAl2008}
Yi~Li, Eric Perlman, Minping Wan, Yunke Yang, Charles Meneveau, Randal Burns,
  Shiyi Chen, Alexander Szalay, and Gregory Eyink, \emph{A public turbulence
  database cluster and applications to study {L}agrangian evolution of velocity
  increments in turbulence}, Journal of Turbulence (2008), no.~9, N31.

\bibitem{Luneburg1966}
Rudolf~Karl Luneburg, \emph{Mathematical theory of optics}, Univ of California
  Press, 1966.

\bibitem{Naimark1968}
Mark~Aronovich Naimark, \emph{Linear differential operators}, vol.~1, Frederick
  Ungar Publishing Co., 1967.

\bibitem{NazaikinskiiEtAl2011}
Vladimir~E Nazaikinskii, Victor~E Shatalov, and Boris~Yu Sternin, \emph{Methods
  of noncommutative analysis: theory and applications}, vol.~22, Walter de
  Gruyter, 2011.

\bibitem{PerlmanEtAl2007}
Eric Perlman, Randal Burns, Yi~Li, and Charles Meneveau, \emph{Data exploration
  of turbulence simulations using a database cluster}, Proceedings of the 2007
  ACM/IEEE Conference on Supercomputing, 2007, pp.~1--11.

\bibitem{Schwarz1870}
H~A Schwarz, \emph{{\"U}ber einen {G}renz{\"u}bergang durch alternierendes
  {V}erfahren}, Vierteljahrsschrift der Naturforschenden Gesellschaft in
  Z{\"u}rich \textbf{15} (1870), 272--286.

\bibitem{Taylor1981}
Michael~E Taylor, \emph{{Pseudodifferential Operators (PMS-34)}}, Princeton
  University Press, Princeton, 1981.

\bibitem{Taylor1991}
\bysame, \emph{{Pseudodifferential Operators and Nonlinear PDE}}, Progress in
  mathematics, Birkhauser, 1991.

\bibitem{WangEtAl2011}
Shen Wang, Maarten~V de~Hoop, and Jianlin Xia, \emph{On 3{D} modeling of
  seismic wave propagation via a structured parallel multifrontal direct
  {H}elmholtz solver}, Geophysical Prospecting \textbf{59} (2011),
  no.~Modelling Methods for Geophysical Imaging: Trends and Perspectives,
  857--873.

\bibitem{Wong2014}
Man-Wah Wong, \emph{An introduction to pseudo-differential operators}, 3rd ed.,
  vol.~6, World Scientific Publishing Company, 2014.

\bibitem{ZhaoEtAl2023}
Liuxian Zhao, Chuanxing Bi, Haihong Huang, Qimin Liu, and Zhenhua Tian, \emph{A
  review of acoustic {L}uneburg lens: {P}hysics and applications}, Mechanical
  Systems and Signal Processing \textbf{199} (2023), 110468.

\end{thebibliography}

\end{document}